\documentclass[12pt,a4paper]{article}
\usepackage[latin1]{inputenc}
\usepackage[left=2.5cm, right=2.5cm, top=2.5cm,bottom=2.5cm]{geometry}
\usepackage{amsmath,amsthm,amsfonts,amssymb}
\usepackage{color,url}
\usepackage{epsfig}
\usepackage{mathrsfs}
\usepackage{enumitem}
\usepackage{graphicx}
\usepackage{lineno}
\usepackage{tikz}
\usepackage{tkz-graph}

\usepackage{amsmath}
\usepackage{xcolor}
\usepackage{multicol}

\newtheorem{theorem}{Theorem}[section]
\newtheorem{proposition}[theorem]{Proposition}
\newtheorem{example}[theorem]{Example}
\newtheorem{lemma}[theorem]{Lemma}

\newtheorem{corollary}[theorem]{Corollary}

\newtheorem{problem}[theorem]{Problem}
\title{Convex Geometries yielded by Transit Functions}

\author{Manoj Changat$^1$\and Lekshmi Kamal K. Sheela$^1$\and  Iztok Peterin$^{2,3}$\and  Ameera Vaheeda Shanavas$^1$}
\date{\today \\$ $\\
	\small	$^1$ Department of Futures Studies, University of Kerala, Thiruvananthapuram - 695581, India. E-mails:  mchangat@keralauniversity.ac.in, lekshmisanthoshgr@gmail.com, ameerasv@gmail.com\\
	\small	$^2$ University of Maribor, FEECS, Koro\v{s}ka 46, 2000 Maribor, Slovenia.\\ 
	\small	$^3$ Institute of Mathematics, Physics and Machanics, Jadranska 19, 1000 Ljubljana, Slovenia\\E-mail: iztok.peterin@um.si
	}

\begin{document}

\maketitle	

\begin{abstract}
Let $V$ be a finite nonempty set. A transit function is a map $R:V\times V\rightarrow 2^V$ such that $R(u,u)=\{u\}$, $R(u,v)=R(v,u)$ and $u\in R(u,v)$ hold for every $u,v\in V$. A set $K\subseteq V$ is $R$-convex if $R(u,v)\subset K$ for every $u,v\in K$ and all $R$-convex subsets of $V$ form  a convexity $\mathcal{C}_R$. We consider Minkowski-Krein-Milman property that every $R$-convex set $K$ in a convexity $\mathcal{C}_R$ is the convex hull of the set of extreme points of $K$ from axiomatic point of view and present a characterization of it. Later we consider several well-known transit functions on graphs and present the use of the mentioned characterizations on them. 
\end{abstract}

\noindent \textbf{Key words}: Minkowski-Krein-Milman property, convexity, convex geometry, transit function\bigskip

\noindent \textbf{AMS subject classification (2020)}: 52A01, 05C38.

\section{Introduction and 
 Preliminaries}\label{introduction}
%\section{Preliminaries}\label{Prel}
The notion of convexity has been extended from Euclidean spaces to several different mathematical structures, and in the last decades an abstract theory of convex structures has been developed. It is based on natural conditions or axioms imposed on a family of subsets $\mathcal{C}$ of a given set $V$. The two axioms are both $\emptyset$ and $V$ belong to $\mathcal{C}$ and $\mathcal{C}$ is closed under arbitrary intersections. In the infinite case, a nested sequence of members of $\mathcal{C}$ should also belong to $\mathcal{C}$. The elements in $\mathcal{C}$ are called the convex sets of $V$. Given a convexity $\mathcal{C}$ on $V$, the \emph{convex hull} of a subset $A$ of $V$ is the smallest convex set containing $A$, denoted as $\langle A \rangle_{\mathcal{C}}$. The theory is comprehensively surveyed in van de Vel's book \cite{VdV}, where the notion of interval convexity turns out to be one of its universal concepts. \\

As shown by Calder in \cite{calder}, the pertinent properties of a mapping $R: V \times V\rightarrow 2^V$ that yields a convexity on a set $V$ are the symmetry law and the extensive law. These two laws constitute all axioms for the formal interval convexity with $R$ as the interval operator \cite{VdV}. By also adding the idempotent law, we get the concept of a transit function as defined by Mulder \cite{muld-08}. His purpose was to generalize natural intervals and convex sets in several mathematical structures, including vector spaces, graphs, posets, lattices, hypergraphs, etc. \\

%Since we are discussing convex geometries through transit functions, 
We formally define a \emph{transit function} on a finite non-empty set $V$ as a map $R: V\times V\rightarrow 2^{V}$ satisfying the following three axioms.
	
$(t1)$ $u \in R(u,v)$, for every $u,v \in V$. 

$(t2)$ $R(u,v) = R(v,u)$, for every $u,v \in V$.

$(t3)$ $R(u,u) = \{u\}$, for every $u \in V$.

The notation $x \in R(u,v)$ can be interpreted as $x$ is between $u$ and $v$, and so an axiom on $R$ is termed a betweenness axiom (also we sometimes use the term transit axiom). We refer to sets $R(u,v)$ as transit sets. 

Given a transit function $R$ on $V$  a set $S\subseteq V$ is said to be \textit{$R$-convex} if $R(u,v)\subseteq S$ for every $u,v\in S$. It can be easily seen that the family of all $R$ convex sets of $V$, known as $R$-convexity and denoted as $\mathcal{C}_R$, forms a convexity on $V$ and we say that the $R$-convexity is generated by the transit function $R$. \\

Furthermore, for a positive integer $n$, a convexity $\mathcal{C}$ on $V$ is of arity at most $n$, if
$\mathcal{C} = \{C\subseteq V: F \subseteq C, |F|\le n \Rightarrow \langle F \rangle_{\mathcal{C}} \subseteq C\}$. It follows that every convexity $\mathcal{C}$ is an $n$-ary convexity for some integer $n\ge 2$, see van De Vel \cite{VdV} and \cite{Chma}, for more discussion on $n$-ary convexities and the $n$-ary transit function. 
Natural convexities $\mathcal{C}$ are interval convexities generated by a transit function $R$; see \cite{VdV}, for a comprehensive axiomatic treatment of interval convexities. \\

Given a convex set $K$ in a convexity $\mathcal{C}$, a point $k\in K$ is called an \textit{extreme point}, if the set $K\setminus \{k\}$ is also convex.
An interesting property satisfied by the classical Euclidean convexity in $\mathbb{R}^d$ is the well-known Minkowski-Krein-Milman theorem, which states that every compact convex set is the convex hull of its extreme points. 
This classical theorem is generalized in the axiomatic convexity and framed as the axiom, namely, \textit{"Every convex set $K$ in a convexity $\mathcal{C}$ is the convex hull of the set of extreme points of $K$"}. The convexity $\mathcal{C}$ on $V$ that satisfies this condition is termed a \textit{convex geometry} on $V$. \\

The theory of convex geometry, due to Edelman and Jamison \cite{edel-jami}, is developed as an abstraction
of standard convexity in Euclidean space $\mathbb{R}^d$. There are several other equivalent
formulations of convex geometries following the pioneering work of Edelman and Jamison, for example, the surveys by Stern~ \cite{Stern}, Adaricheva and C\'{z}edli \cite{adcz}, C\'{z}edli \cite{cz}, and Adaricheva and Nation \cite{adna}. There are many settings where convex geometries naturally appear (under
various names): see, for instance, Goecke, Korte and Lo\'{v}asz \cite{goko}, Korte, Lo\'{v}asz, and Schrader \cite{kolo}, Doignon and Falmagne \cite{dofa}, and Falmagne and Doignon \cite{fado}. For nice historical overviews with reference to
additional settings, see Monjardet \cite{mo08,mo90,mo85}.\\

In \cite{edel-jami}, the \emph{anti-exchange} axiom is introduced as an equivalent condition for a convex geometry. That is, a convexity $\mathcal{C}$ on $V$ satisfies the anti-exchange axiom, if for any convex set $K \in \mathcal{C} $, and two elements $p,q \notin K$, $p\neq q$,
then $q \in \left\langle K \cup \{p\}\right\rangle$ implies that $p \notin  \left\langle K \cup \{q\}\right\rangle$. For convenience, we use the notation $\left\langle K \cup p\right\rangle$ instead of $\left\langle K \cup \{p\}\right\rangle$.\\

 Edelman and Jamison in \cite{edel-jami} have given several equivalent conditions for a convexity $\mathcal{C}$ on $V$ to be a convex geometry. We quote some of the equivalent conditions that we use in this paper as the following Theorem. 
 
 \begin{theorem}\cite{edel-jami}\label{thm:edeljami}
     Let $V$ be a set and $\mathcal{C}$ be a convexity on $V$. The following conditions are equivalent.
     \begin{enumerate}
         \item[(i)] $\mathcal{C}$ is a convex geometry.
         \item[(ii)] $\mathcal{C}$ satisfies the anti-exchange axiom.
         \item[(iii)] For every convex set $K$, there exists a point $p \in V\setminus K$ such that $K\cup \{p\}$ is convex.   
     \end{enumerate}
 \end{theorem}

In his work \cite{muld-08}, Mulder introduced several natural betweenness axioms for a transit function $R$. We require some of these axioms in our results in this paper, which are as follows.

$(b1)$ If $x\in R(u, v)$ and $x\neq v$, then $v\not\in R(u,x)$, for every $u,v,x\in V$. 

$(b3)$ If $x\in R(u,v)$ and $y\in R(u,x)$, then $x\in R(y,v)$, for every $u,v,x,y\in V$.

$(m)$ If $x,y\in R(u, v)$, then $R(x,y)\subseteq R(u,v)$, for every $u,v,x,y\in V$.

It can be verified that, if $R$ satisfies $(b3)$, then it satisfies $(b1)$, but not conversely. The axiom $(m)$ is known as the \emph{monotone axiom} and if $(m)$ holds, then $R(u,v)$ is $R$-convex for every $u,v\in V$.

It may be noted that a transit function $R$ can also be equivalently represented as a ternary relation and that the transit axioms can be translated as axioms on the ternary relation. 

In 2009, Chv\'{a}tal \cite{chva} defined convex geometries on $V$ using the betweenness relation $B$ considered as ternary relations with more relaxed axioms and tried to characterize some special classes of convex geometries using this approach.  

The relation of betweenness of Chv\'{a}tal can be translated as a function $B$ from $V\times V \rightarrow 2^V$ that satisfies $u \notin B(u,v)$ and $B(u,v) = B(v,u)$, for all $u,v \in V$. 
 
Let $X\subseteq V$ that is not necessarily convex and $x\in X$. If $x\notin B(u,v)$ for every points $u,v\in X-\{x\}$, then $x$ is an \emph{extreme point} of $X$. The set of all extreme points of $X$ is denoted by $ex_B(X)$. Using the idea of the set of extreme points, the results of Chv\'{a}tal in \cite{chva} can be stated as follows.

\begin{theorem}\cite{chva}\label{chav}
 For every betweenness $B$ on a finite $V$ with $x\in B(u,v)$ and pairwise distinct $u,x,v\in V$, the following statements are  equivalent.
\begin{description}
\item[$(eB1)$] For any subset $X$ of $V$ and every $x_1, x_2,x_3\in X$ such that $x_2 \in B(x_1,x_3)$
there exists $\overline{x}_1, \overline{x}_3\in ex_B(X) $  such that $x_2 \in B(\overline{x}_1, \overline{x}_3)$.
\item[$(eB2)$] For every $x,y,u,v,w\in V$, if $x\in B(u,v)$ and $y\in B(x,w)$, then $y\in B(u,w)$ or $y\in B(v,w)$ or $y\in B(u,v)$.
\end{description}
\end{theorem}

The implication of Theorem~\ref{chav} is that the condition $(eB2)$ implies that the corresponding betweenness $B$ will give rise to a convex geometry. 

The work on convex geometries due to Chv\'{a}tal through betweenness prompts one to think about how far the idea of betweenness can capture abstract convex geometries. Since the transit function is an important tool for studying betweenness, as evidenced by the vast references in discrete structures, \cite{chklmu-01,mcjmhm-10,Changat:19a,Changat-20,Changat:21,Changat:18a,Changat-22,Chvatal-11,momu-02,mune-09,VdV}, the following problem becomes interesting. 

\noindent\textbf{Problem.} \textit{Is it possible to characterize convex geometries arising from interval convexities $\mathcal{C}_R$ using the betweenness axioms of the corresponding transit function $R$?}\medskip

In this paper, we partially solve
this problem and provide connections to the well-known instances of convex geometries which are interval convexities.\\

We observe that condition $(eB2)$ in Theorem~\ref{chav} is crucial for a transit function $R$ to generate convex geometries. Therefore we rename it to $(Ch)$ after the author.

$(Ch)$ For every $u,v,w,x,y\in V$, if $x\in R(u,v)$ and $y\in R(x,w)$, then $y\in R(u,w)$ or $y\in R(v,w)$ or $y\in R(u,v)$.

The axiom $(J0)$, discussed in \cite{Changat-22} and several other references, turns out to be a necessary axiom for a transit function $R$ to generate convexities that are convex geometries. The axiom $(J0)$ is the following.

$(J0)$ For every distinct $u,v,x,y\in V$, if $x\in R(u,y)$ and $y\in R(x, v)$, then $x\in R(u,v)$.   
	
In this paper, we consider only finite non-empty sets $V$.
We organize the paper as follows.  
In Section~\ref{convex-transit}, we present our main results for an arbitrary transit function to generate interval convexities that are convex geometries and provide characterization in special cases. In Section~\ref{graph-convex geometries}, we provide examples of convexities in graphs, hypergraphs, etc. generated by transit functions and analyze the status of the axioms presented in Section~\ref{convex-transit} on these transit functions. In Section~\ref{identified}, we study convex geometries from the point of view of set systems. We characterize transit functions that identify convex geometries.

\section{Convex geometry generated by axioms on an arbitrary transit function $R$}\label{convex-transit}

In this section, we present the main results using axioms for an arbitrary transit function $R$ such that the corresponding $R$-convexity $\mathcal{C}_R$ is a convex geometry. 
It can be observed that $(b1)$  forms a necessary axiom for a transit function $R$ to generate $\mathcal{C}_R$ a convex geometry.
We have the following immediate theorem.

\begin{theorem} \label{Cg1}
Let $R$ be a monotone transit function on a finite non-empty set $V$. If $\mathcal{C}_R$ is a convex geometry, then $R$ satisfies $(b1)$ and $(J0)$.
\end{theorem}

\begin{proof}
Assume that $\mathcal{C}_R$ is a convex geometry generated by a monotone transit function $R$. For a convex set $K \in \mathcal{C_R}$ and two unequal points $p$ and $q$ in $V\setminus K$, if $ q \in \left\langle K \cup p\right\rangle$, then it follows that $p \notin  \left\langle K \cup q\right\rangle$. First, we prove that $R$ satisfies $(b1)$. If not, then $x\in R(u,v)$, $ x\neq v$ and $v\in R(u,x)$. For a convex set $K=\{u\}$ we have $x\in\left\langle K\cup v\right\rangle$ and $v\in\left\langle K\cup x\right\rangle$, a contradiction with the assumption that $\mathcal{C}_R$ is a convex geometry. 

Now suppose that $R$ does not satisfy $(J0)$. That is, $x\in R(u,y)$, $y \in R(x,v)$ and $x\notin R(u,v)$ for some different points $u,x,y,v$. If $y\in R(u,v)$, then $x\in R(u,y)\subseteq R(u,v)$ since $R$ is monotone, a contradiction. So, $y\notin R(u,v)$. Since $R$ is monotone, $R(u,v)$ is convex and let $K=R(u,v)$.  Then $x \in \left\langle K\cup y\right\rangle$ since $x\in R(u,y)$ and $y \in \left\langle K\cup x\right\rangle$ since $y\in R(x,v)$. That is, $x \in \left\langle K\cup y\right\rangle$ and $y \in \left\langle K\cup x\right\rangle$, which contradicts $(ii)$ of Theorem \ref{thm:edeljami} that $\mathcal{C}_R$ is a convex geometry. Therefore, we can conclude that $R$ satisfies $(J0)$. 
\end{proof} 
 
It is clear from the proof of Theorem~\ref{Cg1} that if a convexity generated by a (not necessary monotone) transit function is a convex geometry, then $R$ satisfies $(b1)$, but it does not need to satisfy $(J0)$.  As an illustration, consider Example~\ref{not m}, where $\mathcal{C}_R$ constitutes a convex geometry and $R$ satisfy $(b1)$ but not $(J0)$ nor $(m)$.

\begin{example}\label{not m}
Let $V=\{a,b,c,d,e\}$ and  $R: V\times V\rightarrow 2^V$ be defined symmetrically by $R(a,c)=\{ a,b,c\}$, $R(a,d)=\{ a,b,d\}$, $R(a,e)=\{ a,b,d,e\}$, $R(d,e)=\{d,c,e\}$, $R(b,d)=\{b, c, d\}$  and in all other cases $R(x,y)=\{x,y\}$. Here $\mathcal{C}_R$ is a convex geometry and $R$ satisfies $(b1)$. But $R$ does not satisfy $(J0)$ as $c \in R(d,b)$, $b \in R(c,a)$ and $c \notin R(d,a)$. Also, it is not monotone as $b,d \in R(a,e)$ but $R(b,d) \not\subseteq R(a,e) $.
\end{example}

Consider the following example in which $R$ is a monotone transit function that satisfies $(b1)$ and $(J0)$, but $\mathcal{C}_R$ is not a convex geometry.

\begin{example}\label{ncg}
Let $V=\{a,b,c,d,e,f\}$ and $R: V\times V\rightarrow 2^V$ be defined symmetrically by $R(a,f)=\{ a,c,f\}$, $R(c,e)=\{c,b,e\}$, $R(b,d)=\{b, a, d\}$  and in all other cases $R(x,y)=\{x,y\}$. Then $R$  is monotone and $R$ satisfies $(J0)$ and $(b1)$. But $\mathcal{C}_R$ is not a convex geometry by Theorem \ref{thm:edeljami}, as anti-exchange axiom does not hold, because $K= \{d,e,f\}$ is a convex set and $c\in \left\langle K\cup a\right\rangle$ and $a\in \left\langle K\cup c\right\rangle$ (since $b \in R(c,e)$ and $a \in R(b,d)$).
\end{example}

Example~\ref{ncg} demonstrates that monotonicity, $(J0)$ and $(b1)$ are not sufficient for a convexity generated by a transit function $R$ to become a convex geometry. The next theorem where we replace monotonicity by $(Ch)$ gives a sufficient condition for a convexity induced by a transit function to be a convex geometry. Before that we show the following technical lemma. 

\begin{lemma}\label{useful}
    Let $R$ be a transit function on a finite non-empty set $V$ that fulfills $(Ch)$. If $K$ is an $R$-convex set, $p,q\notin K$ and $p\in \left\langle K\cup q\right\rangle$, then $p\in R(q,x)$ for some $x\in K$. 
\end{lemma}

\begin{proof}
Let $K$ be an $R$-convex set, $p,q\notin K$ and $p\in \left\langle K\cup q\right\rangle$. This means that there exists two sequences of points $x_1,\dots,x_{t-1}\in K$ and $y_1,\dots,y_t\in V-K$, such that $y_1=q$, $y_t=p$ and $y_{i+1}\in R(y_i,x_i)$ for every $i\in\{1,\dots, t-1\}$. We will show that $t=2$ and with this $p=y_2\in R(x_1,y_1)=R(x_1,q)$ for some $x_1\in K$. We may choose minimum $t$. If $t>2$, then $y_2\in R(y_1,x_1)$ and $y_3\in R(y_2,x_2)$ where $x_1,x_2\in K$ and $y_1,y_2,y_3\in V-K$. By $(Ch)$ we have $y_3\in R(x_1,x_2)$ or $y_3\in R(y_1,x_1)$ or $y_3\in R(y_1,x_2)$. If $y_3\in R(x_1,x_2)$, then we have a contradiction because $K$ is $R$-convex. If $y_3\in R(y_1,x_1)$, then we have two shorter sequences by removing $x_2$ and $y_2$ from original sequences, which is in contradiction with the choice of $t$. Finally, if  $y_3\in R(y_1,x_2)$, then we have two shorter sequences by removing $x_1$ and $y_2$ from original sequences, the same contradiction with the choice of $t$. Hence, $t=2$ and we are done. 
\end{proof}
   
\begin{theorem}\label{Cg2}
    Let $R$ be a transit function on a finite non-empty set $V$. If $R$ satisfies $(Ch)$, $(b1)$ and $(J0)$, then $\mathcal{C}_R$ is a convex geometry. 
\end{theorem}

\begin{proof}
Suppose that $R$ is a transit function that satisfies $(Ch)$, $(b1)$ and $(J0)$. To prove that $\mathcal{C}_R$ is a convex geometry, it suffices to prove that the anti-exchange axiom holds by Theorem \ref{thm:edeljami}. This means that $p \in  \left\langle K \cup q\right\rangle$ implies $q \notin \left\langle K \cup p\right\rangle$ for any convex set $K \in \mathcal{C}$ and two different elements $p,q \notin K$.

Let $p \in  \left\langle K \cup q\right\rangle$. By Lemma \ref{useful} we have $p\in R(q,x)$ for some $x\in K$. Assume conversely, that also $q\in \left\langle K\cup p\right\rangle$. By Lemma \ref{useful} again we have $q\in R(p,y)$ for some $y\in K$. Since $p\in R(q,x)$, $q\notin R(p,x)$ by $(b1)$ so $y\neq x$. Now $p\in R(x,q)$ and $q\in R(p,y)$ implies that $p,q\in R(x,y)$ by $(J0)$, contradicting the assumption that $K$ is a convex set. So $q\notin \left\langle K\cup p\right\rangle$ and we are done. 
\end{proof}

Examples~\ref{not $(J0)$}, \ref{not $(Ch)$}, and \ref{not $(b1)$} illustrate the independence of $(Ch)$, $(b1)$, and $(J0)$. Also, the $R$-convexity of these three examples is not a convex geometry. %This observation underscores the requirement that for $\mathcal{C}_R$, generated by $R$ in these examples to be a convex geometry, it must satisfy all three axioms.

\begin{example}\label{not $(J0)$} $(Ch)$ and $(b1)$, but not $(J0)$.\\
Let $V=\{a,b,c,d\}$ and  $R: V\times V\rightarrow 2^V$ be defined symmetrically by $R(a,c)=\{ a,b,c\}$,  $R(b,d)=\{ b,c,d\}$ and in all other cases $R(x,y)=\{x,y\}$. Then $R$ satisfies $(Ch)$ and $(b1)$. Here $b\in R(a,c)$ and  $c\in R(b,d)$ but  $b\notin R(a,d)$ and so $R$ do not satisfy $(J0)$.
 \end{example}
 
\begin{example}\label{not $(Ch)$} $(J0)$ and $(b1)$, but not $(Ch)$.\\
Let $V=\{a,b,c,d,e\}$ and $R: V\times V\rightarrow 2^V$ be defined symmetrically by $R(a,e)=\{ a,c, d,e\}$,  $R(c,d)=\{ c,b,d\}$,  $R(a,b)=\{ a,e,b\}$ and in all other cases $R(x,y)=\{x,y\}$. Then $R$ satisfies $(J0)$ and $(b1)$. Here $c\in R(a,e)$ and  $b\in R(c,d)$ but  $b\notin R(a,d)$, $b\notin R(a,e)$ and $b\notin R(d,e)$  and so $R$ does not satisfies $(Ch)$.
 \end{example}

\begin{example}\label{not $(b1)$} $(J0)$ and $(Ch)$, but not $(b1)$.\\
Let $V=\{a,b,c\}$ and $R: V\times V\rightarrow 2^V$ be defined symmetrically by $R(a,c)=\{ a,b,c\}$,  $R(b,c)=\{ b,a,c\}$ and  $R(a,b)=\{ a,b\}$. Then $R$ satisfies $(J0)$ and $(Ch)$. But $b\in R(a,c)$ and  $a\in R(b,c)$  so $R$ does not satisfy $(b1)$.
\end{example}

\begin{lemma}\label{Cg implies m}
    Let $R$ be a transit function on a set $V$. If $R$ satisfies $(Ch)$, then $R$ is monotone.
\end{lemma}
\begin{proof}
    Let $R$ be a transit function that satisfies $(Ch)$. Assume $x,y \in R(u,v)$ and $z \in R(x,y)$. Now $x \in R(u,v)$ and $z \in R(x,y)$ implies that $z \in R(u,y)$ or $z \in R(v,y)$ or $z \in R(u,v)$. If $z \in R(u,v)$, then we are done. Otherwise $z \in R(u,y)$ or $z \in R(v,y)$. First consider $z \in R(u,y)$. Now, $(Ch)$ together with $y \in R(u,v)$  imply that $z \in R(u,u)=\{u\}\subseteq R(u,v)$ or $z \in R(u,v)$ and we are done. Let now $z \in R(v,y)$. Again, $(Ch)$ together with $y \in R(u,v)$  imply that $z \in R(v,v)=\{v\}\subseteq R(u,v)$ or $z \in R(u,v)$ and we are done. 
\end{proof}

 By Theorems~\ref{Cg1} and \ref{Cg2} and Lemma~\ref{Cg implies m}  we obtain the following result. 
 
\begin{theorem}\label{cg}
Let $R$ be a transit function on a finite nonempty set $V$ satisfying $(Ch)$. Then $\mathcal{C}_R$ is a convex geometry if and only if $R$ satisfies $(b1)$ and $(J0)$. 
\end{theorem}

The Peano axiom $(P)$ is adapted from the well-known geometric property in Euclidean space. It was considered by van De Vel in \cite{VdV}, as well as many others, see \cite{Changat-20} for a flavor. It is defined as follows. 

$(P)$ For every $x,y,u,v,w\in V$, if $x\in R(u,v)$ and $y\in R(x,w)$, then there exists $z\in R(u,w)$ such that $y\in R(z,v)$.

The next result shows that $(Ch)$ is stronger than $P$, while the converse  does not hold in every case as shown later in Example \ref{nice}.

\begin{lemma}
If a transit function $R$ satisfies $(Ch)$, then it satisfies $(P)$.
\end{lemma}
\begin{proof}
 Suppose that $R$ satisfies $(Ch)$. So, if $x\in R(u,v)$ and $y\in R(x,w)$, then $y\in R(u,w)$ or $y\in R(v,w)$ or $y\in R(u,v)$. If $y\in R(u,w)$, then assume $z=y$ and so $y\in R(z,v)$. If $y\in R(v,w)$, then set $z=w$ and we get $y\in R(z,v)$. If $y\in R(u,v)$, then for $z=u$ we have $y\in R(z,v)$. 
\end{proof}

In Theorem~\ref{Cg2}, we can substitute $(Ch)$ with the weaker $(P)$. For this, we need the following propositions proved in \cite{VdV}. A convexity is said to be \emph{join hull commutative (JHC)} if for any convex set $K\subset V$ and for any point $p$ in $V$, $\left\langle K\cup p\right\rangle = \cup \{\left\langle k\cup p\right\rangle: k\in K\}$. If $R$ is a transit function on $V$,  then we call the transit function $R^*:V\times V\rightarrow 2^{V}$ defined by $R^*(u, v) = \left\langle R(u,v) \right\rangle$, for $u, v \in  V$, the \emph{segment
transit function} associated with R.
%Now consider the following proposition which was proved in \cite{VdV}.

\begin{proposition}\cite{VdV}\label{pe im jhc}
If $R$ is a transit function, then $R$-convexity is JHC if and only if $R^*$ satisfies $(P)$.
\end{proposition}
\begin{proposition}\cite{VdV}\label{pe im mo}
If a transit function $R$ satisfies $(P)$, then $R$  is monotone.
\end{proposition}

\begin{theorem}
If $R$ is a transit function satisfying $(P)$, $(J0)$ and $(b1)$, then $\mathcal{C}_R$ is a convex geometry.
\end{theorem}

\begin{proof}
 By Proposition~\ref{pe im mo}, $R$ is monotone since it satisfies $(P)$. Since $R$ is monotone, we have $R^*(u,v)=R(u,v)$ and $(P)$ holds for $R^*$. By Proposition~\ref{pe im jhc}, $\mathcal{C}_R$  is JHC. If  $\mathcal{C}_R$ is JHC and $R$ is monotone, then $p\in \left\langle K\cup q\right\rangle$ implies $p \in R(q,x)$ for an $R$-convex set $K$, $p,q\notin K$ and some $x\in K$. Now we may assume that $\mathcal{C}_R$ is not a convex geometry. That is, $p\in \left\langle K\cup q\right\rangle$ and $q\in \left\langle K\cup p\right\rangle$ for some $R$-convex set $K$. Since $\mathcal{C}_R$ is JHC, $p\in \left\langle K\cup q\right\rangle$ implies that $p \in R(q,x)$ for some $x \in K$. Furthermore, $q\in \left\langle K\cup p\right\rangle$ implies that $q \in R(p,y)$ for some $y \in K$. If $x=y$, then we have a contradiction to $(b1)$ because $p \in R(q,x)$ and $q \in R(p,x)$. Now $p \in R(q,x)$ and $q \in R(p,y)$ and $p,q,x,y$ are different, which implies that $p\in R(x,y)$ by $(J0)$, a contradiction to the assumption that $K$ is an $R$-convex set. So $\mathcal{C}_R$ is a convex geometry. 
\end{proof}

Now, we can replace $(Ch)$ with the weaker $(P)$ in Theorem~\ref{cg} to get more classes of interval convex geometries, which can be obtained from axioms on the corresponding transit functions.

\begin{theorem}\label{Peano-ch}
Let $R$ be a transit function on a finite nonempty set $V$ satisfying $(P)$. Then $\mathcal{C}_R$ is a convex geometry if and only if $R$ satisfies $(b1)$ and $(J0)$. 
\end{theorem}

We end this section by presenting an example of a transit function $R$ whose $R$-convexity is a convex geometry and $R$ satisfies $(P)$, $(J0)$ and $(b1)$, but not $(Ch)$. For clarity, the example is shown in Figure~\ref{fig:CG-Ch} below. 
   
\begin{example}\label{nice}
Let $V=\{u,x,v,y,w,s,t\}$ and  $R: V\times V\rightarrow 2^V$ be defined symmetrically by $R(u,v)=\{ u,x,v\}$, $R(x,w)=\{ x,y,w\}$, $R(u,w)=\{ u,s,w\}$, $R(v,w)=\{v,t,w\}$, $R(u,t)=\{u,y,t\}$, $R(v,s)=\{v,y,s\}$  and in all other cases $R(a,b)=\{a,b\}$. It is easy to check that $R$ satisfies $(P)$, $(b1)$ and $(J0)$. But $R$ does not satisfy $(Ch)$ as $x \in R(u,v)$, $y \in R(x,w)$ but $y \notin R(u,v)$, $y \notin R(u,w)$ and $y \notin R(v,w)$. Further, one can observe that every $R(a,b)$, $a,b\in V$, is convex. From this it is straightforward to verify $(iii)$ of Theorem \ref{thm:edeljami} and $\mathcal{C}_R$ is a convex geometry.  
\end{example}

\begin{figure}[ht]
    \begin{center}
     \includegraphics[height=8cm]{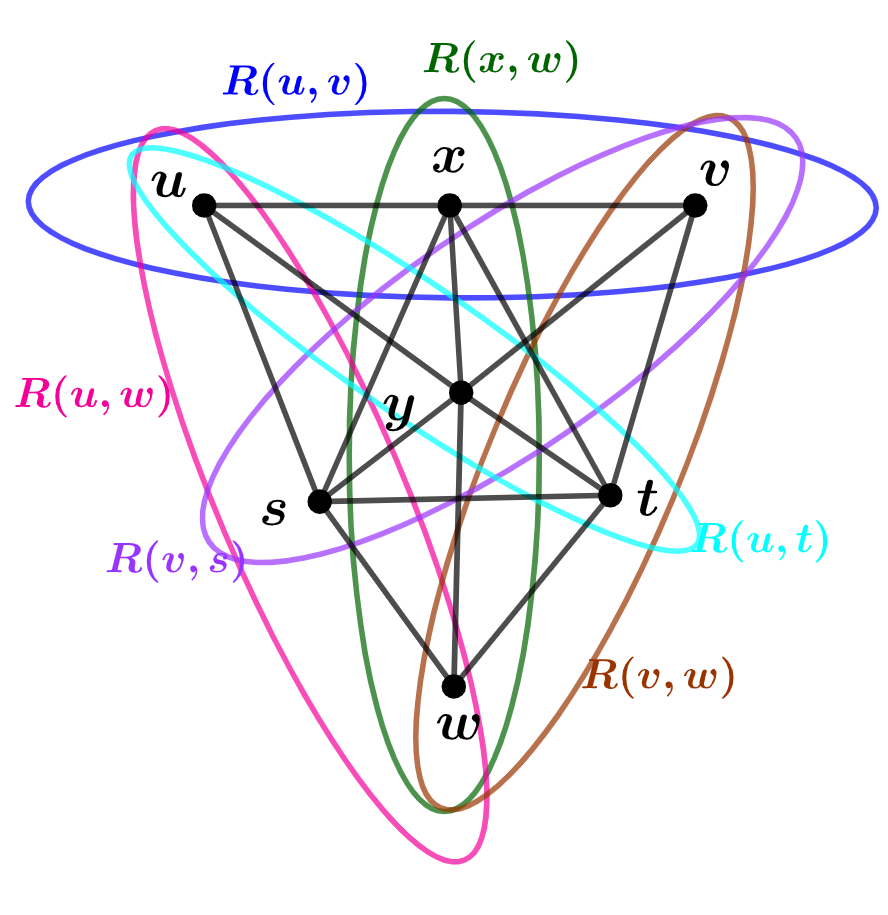} 
    \end{center}
   \caption{$R$ satisfies $(P),(J0)$ and $(b1)$, but not $(Ch)$.}\label{fig:CG-Ch}
\end{figure}

\section{Instances of convex geometries generated by well-known transit functions} \label{graph-convex geometries}

In this section, we analyze the known cases of interval convexities mostly in graphs known to be convex geometries, and check the status of the axioms that we have discussed in the previous sections. We provide each case in subsections. 

We consider only undirected, finite, simple and connected graphs. Let $G$ and $G_1,\ldots,G_k$ be connected graphs. We say that $G$ is $(G_1\cdots G_k)$-free graph if $G$ has no induced subgraph isomorphic to $G_i$ for every $i\in \{1,\dots,k\}$. The forbidden subgraphs that will be treated in this work are holes, denoted by $H$, which are all induced cycles $C_n$, for $n\geq 5$. Other graphs are house $H$ (left graph of Figure \ref{hhd}), domino $D$ (third graph of Figure \ref{hhd}), $A$-graph (last graph of Figure \ref{hhd}) and a $3$-fan $F_3$ that consists of a path on four vertices and additional vertex that is adjacent to all vertices of the mentioned path. Notice that we denote both house and a hole by $H$. In particular, always when they are both forbidden, we write $HH$-free graphs and there is no place for confusion. Otherwise, if there is only $H$-free graph, we mean a house-free graphs. For hole-free graphs we use the whole word: $hole$-free graphs.
   
\subsection{Geodesic convexity}\label{gc}

One of the well-studied path transit functions on graphs is the interval function $I_G$ with respect to the standard shortest path distance $d(u,v)$, $u,v\in V(G)$ in $G$. Here is $d(u,v)$ the minimum number of edges on a $u,v$-path and a $u,v$-path of length $d(u,v)$ is called a $u,v$-\textit{shortest path} or a $u,v$-\textit{geodesic}. The \emph{interval function} $I_G$ of a connected graph $G$ is defined  as $I: V\times V \longrightarrow 2^{V}$ where
	 	$$I_G(u,v)=\{w\in V(G): w \text{ lies on some }u,v\text{-geodesic in }G \},$$
which can also be expressed by the distances as
		$$I_G(u,v)=\{w\in V(G): d(u,w) +d(w,v) = d(u,v) \}.$$
  In \cite{fabe}, it was established that the geodesic convexity of a graph $G$ is a convex geometry if and only if $G$ is a Ptolemaic graph. \emph{Ptolemaic graphs} were introduced by Kay and Chartrand in \cite{kay} as graphs in which the distances obey the Ptolemy inequality. That is, for every four vertices $u, v, w$ and $x$ the inequality $d(u,v)d(w,x) + d(u,x)d(v,w) \geq d(u,w)d(v,x)$ holds.  Howorka in \cite{edwa} proved that a graph is Ptolemaic if and only if it is chordal and $3$-fan-free. A graph is chordal if it has no induced cycles $C_n$ for $n \geq 4$
\begin{theorem}\cite{fabe}\label{geodesic-convex_geometry}
The geodesic convexity of a graph G is a convex geometry if and only if $G$ is Ptolemaic.
\end{theorem}

An alternative proof of Theorem \ref{geodesic-convex_geometry} follows from Theorem~\ref{cg} if we are able to show the desired relationships for geodesic convexity between $(J0),(b1),(Ch)$ and Ptolemaic graphs. This was proved for $(J0)$ in \cite{Changat-22} as follows. 

\begin{theorem}\cite{Changat-22}
Let $G$ be a graph. The interval function $I_G$ satisfies $(J0)$ if and only if $G$ is a Ptolemaic graph.
\end{theorem}

Further, Mulder's work \cite{muld-80} has already confirmed that $I$ satisfies $(b1)$ on every connected graph $G$. So, we only need to demonstrate that $I$ satisfies $(Ch)$ on Ptolemaic graphs. 

\begin{proposition}\label{ChodalCh}
The interval function $I$ satisfies $(Ch)$ on Ptolemaic graphs.
\end{proposition}

\begin{proof}
Let $G$ be a Ptolemaic graph (chordal and 3-fan-free) and suppose that $I$ does not satisfy $(Ch)$. That is, $x\in I(u,v)$, $y\in I(x,w)$, $y\notin I(u,v)$, $y\notin I(u,w)$ and $y\notin I(v,w)$. We may choose vertices $x,y,u,v,w$ in such a way that $d(x,w)$ is minimal. Suppose $P$ is a $u,v$-geodesic containing $x$ and $Q=xx_1\dots x_pw$, $p\geq 1$, is an $x,w$-geodesic containing $y$. Since $y\notin I(u,v)$, $y$ is not on  $P$. Similarly, $y\notin I(u,w)$ and $y\notin I(v,w)$ imply the existence of a $u,w$-geodesic $R$ which does not contain $y$ and a $v,w$-geodesic $S$ that also avoids $y$. Let $a$ be the vertex common to $P$ and $R$ and $a'$ be the vertex common to $Q$ and $R$.  Clearly, $a$ is different from $x$ and $a'$ is different from $y$. Let $u'$ and $v'$ be neighbors of $x$ on $P$, where $u'$ is closer to $u$. If both $u'x_1,v'x_1\in E(G)$, then $x_1\neq y$ since we have a $u,v$-geodesic that contains $x_1$. Moreover, this yields a contradiction with the choice of $d(x,w)$, since we can replace $x$ by $x_1$. So, we may assume that at least one of $u'$ and $v'$ is not adjacent to $x_1$, say $u'x_1\notin E(G)$.  If $u'$ is adjacent to any other vertex $q$ on $Q$ (other than $x$), then $xu'q\xrightarrow{Q}x$ form an induced cycle (if $q$ is chosen to be closest to $x$ among all such vertices) of length at least four, a contradiction. Let $R'$ be a shortest path that starts in $u'$ and ends in a vertex $q'$ of $Q$ that avoids $x$ and let $u_1$ be a neighbor of $u'$ on $R'$. If $u_1$ belongs to $P$, then it is not adjacent to $x$ and $u_1u'xx_1$ form or are a part of an induced cycle of length at least four, a contradiction. Otherwise, to avoid an induced cycle of length at least four, we have $u_1x,u_1x_1\in E(G)$. Now, edge $u_1x_2$ yields a 3-fan on $u_1,u',x,x_1,x_2$, which is not possible for Ptolemaic graphs (notice that $x_2$ may be $a'$ or even $w$). Otherwise, $u_1x_2\notin E(G)$ and $u_1,x_1,x_2,x_3$ form or are a part of an induced cycle of length at least four, a final contradiction. Hence, $(Ch)$ holds for $I$ on Ptolemaic graphs.
\end{proof}

\subsection{Induced path or monophonic convexity}\label{mc}

An \textit{induced path} or a \textit{monophonic path} is a chordless path, where a chord of a path is an edge joining two non-consecutive vertices of that path. The \emph{induced path transit function} $J(u,v)$ of graph $G$ is a natural generalization of the interval function and is defined as follows:
$$J(u, v) =\{w \in V(G) :w\text{ lies on an induced }u,v\text{-path}\}$$
 The induced path convexity (monophonic convexity) is a convex geometry precisely for chordal graphs, see \cite{fabe}.
 
\begin{theorem}\cite{fabe}\label{moncon}
 The monophonic convexity of a graph $G$ is a convex geometry if and only if $G$ is a chordal graph.
\end{theorem}

In \cite{Changat-22}, a characterization of the induced path function of chordal graphs is provided. Furthermore, in \cite{momu-02}, the graphs for which $J$ satisfies $(b1)$ are identified as $(HHD)$-free graphs where $H$-house, $H$-hole, and $D$-domino are depicted in Figure~\ref{hhd}. 

\begin{theorem}\cite{Changat-22}
Let $G$ be a graph. The induced path transit function $J_G$ satisfies the axiom $(J0)$ if and only if $G$ is a chordal graph.
\end{theorem}

\begin{lemma}\label{Jb1}\cite{momu-02}
The induced path transit function $J_G$ on a graph $G$ satisfies $(b1)$ if and only if $G$ is $(HHD)$-free.
\end{lemma}

Since chordal graphs are also $(HHD)$-free graphs, it follows by Lemma~\ref{Jb1} that $J$ satisfies $(b1)$ on chordal graphs. Hence, to complete an alternative proof of Theorem \ref{moncon} we only need to show that $J$ satisfies $(Ch)$ on chordal graphs and we are done by Theorem \ref{cg}.

\begin{proposition}
The induced path function $J$ satisfies $(Ch)$ on chordal graphs.
\end{proposition}
 \begin{proof}
If possible suppose that $J$ does not satisfy $(Ch)$ for chordal graphs. That is, $x\in J(u,v)$, $y\in J(x,w)$,  $y\notin J(u,v)$, $y\notin J(u,w)$, and $y\notin J(v,w)$. Suppose $P$ is an induced $u,v$-path containing $x$ and $Q$ be an induced $x,w$-path containing $y$. Since $y\notin J(u,v)$, $y$ is not on $P$ and $y$ is not adjacent to both $u$ and $v$. Now, $y\notin J(u,w)$ and $y\notin J(v,w)$ implies that there is a chord from $u,x$-subpath of $P$ to $y,w$-subpath of $Q$ and a chord from $v,x$-subpath of $P$ to $y,w$-subpath of $Q$. Let $a$ be the vertex closest to $x$ in the $u,x$-subpath of $P$, which has a neighbor, denoted as $a'$, in the $y,w$-subpath of $Q$. Similarly, let $b$ be the vertex closest to $x$ in the $v,x$-subpath of $P$, which has a neighbor, denoted as $b'$, in the $y,w$-subpath of $Q$. Clearly, $a$ and $b$ can be chosen in such a way that $a'$ and $b'$ are different from $y$ and we may choose $a'$ and $b'$ to be as close as possible to $y$ on $Q$. If $a'= b'$, then the vertices on the $a,b$-subpaths of $P$, along with the vertex $a'$, induce a cycle of length at least four. This contradicts the assumption that $G$ is chordal. If $a'\neq b'$, then the vertices on the $a,b$-subpaths of $P$ along with the vertices on the $a',b'$-subpaths of $Q$ induce a cycle of length more than four, the same contradiction again. 
\end{proof}

\subsection{$m^3$-convexity}

For a graph $G$, an $m^3$-path refers to an induced path of length at least three. The $m^3$-transit function is a path transit function defined as $m^3: V\times V \longrightarrow 2^{V}$ where $m^3(u,v)$ equals to
$$\{w\in V(G): w \text{ lies on some }u,v\text{-induced path of length at least three in }G \}$$ 
The $m^3$-convexity is generated by the $m^3$-transit function.
 In \cite{drni},  it has been shown that weak bipolarizable graphs are the convex geometry with respect to $m^3$-convexity. A graph is \emph{weak bipolarizable} if it is $HHDA$-free, see \cite{ola}, where house $H$, hole $H$, domino $D$, and $A$-graph are depicted in Figure \ref{hhd}. 
 
\begin{theorem}\cite{drni}
The  $m^3$-convexity of a graph $G$ is a convex geometry if and only if  $G$ is weak bipolarizable.
\end{theorem}
 
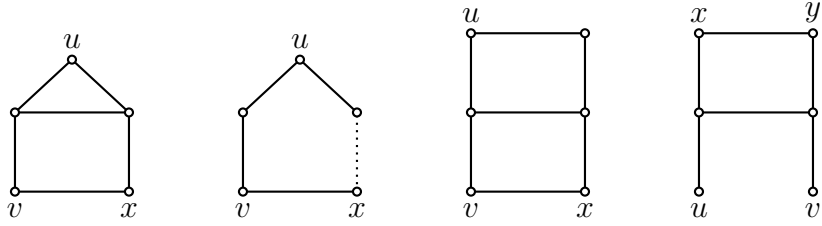
\begin{figure}[ht]
\begin{center}
\begin{tikzpicture}[scale=0.5,style=thick,x=1cm,y=0.7cm]
\def\vr{3pt} % \vr = vertex radius;

% define vertices
%%%%% %%%%% house
\path (0,0) coordinate (a);
\path (3,0) coordinate (b);
\path (3,3) coordinate (c);
\path (0,3) coordinate (d);
\path (1.5,5) coordinate (e);

%  edges
\draw (a) -- (b) -- (c) -- (d) -- (a);
\draw (c) -- (e) -- (d);

\draw (a) [fill=white] circle (\vr);
\draw (b) [fill=white] circle (\vr);
\draw (c) [fill=white] circle (\vr);
\draw (d) [fill=white] circle (\vr);
\draw (e) [fill=white] circle (\vr);

\draw[anchor = north] (a) node {$v$};
\draw[anchor = north] (b) node {$x$};
\draw[anchor = south] (e) node {$u$};

%%%%%%%%%%%%%%%%%%%%C_5

\path (6,0) coordinate (a1);
\path (9,0) coordinate (b1);
\path (9,3) coordinate (c1);
\path (6,3) coordinate (d1);
\path (7.5,5) coordinate (e1);

%  edges
\draw (a1) -- (b1); 
\draw (c1) -- (e1) -- (d1) -- (a1);
\draw[dotted] (b1) -- (c1);
\draw (a1) [fill=white] circle (\vr);
\draw (b1) [fill=white] circle (\vr);
\draw (c1) [fill=white] circle (\vr);
\draw (d1) [fill=white] circle (\vr);
\draw (e1) [fill=white] circle (\vr);

\draw[anchor = north] (a1) node {$v$};
\draw[anchor = north] (b1) node {$x$};
\draw[anchor = south] (e1) node {$u$};

%%%%%%%%%%%%%%%%%%%%%%Domino

\path (12,0) coordinate (a2);
\path (15,0) coordinate (b2);
\path (15,3) coordinate (c2);
\path (12,3) coordinate (d2);
\path (12,6) coordinate (e2);
\path (15,6) coordinate (f2);

%  edges
\draw (a2) -- (b2) -- (c2) -- (d2) -- (a2);
\draw (c2) -- (f2) -- (e2) -- (d2);

\draw (a2) [fill=white] circle (\vr);
\draw (b2) [fill=white] circle (\vr);
\draw (c2) [fill=white] circle (\vr);
\draw (d2) [fill=white] circle (\vr);
\draw (e2) [fill=white] circle (\vr);
\draw (f2) [fill=white] circle (\vr);

\draw[anchor = north] (a2) node {$v$};
\draw[anchor = north] (b2) node {$x$};
\draw[anchor = south] (e2) node {$u$};

%%%%%%%%%%%%%%%%%%%%%%Domino

\path (18,0) coordinate (a3);
\path (21,0) coordinate (b3);
\path (21,3) coordinate (c3);
\path (18,3) coordinate (d3);
\path (18,6) coordinate (e3);
\path (21,6) coordinate (f3);

%  edges
\draw  (b3) -- (c3) -- (d3) -- (a3);
\draw (c3) -- (f3) -- (e3) -- (d3);

\draw (a3) [fill=white] circle (\vr);
\draw (b3) [fill=white] circle (\vr);
\draw (c3) [fill=white] circle (\vr);
\draw (d3) [fill=white] circle (\vr);
\draw (e3) [fill=white] circle (\vr);
\draw (f3) [fill=white] circle (\vr);

\draw[anchor = north] (a3) node {$u$};
\draw[anchor = north] (b3) node {$v$};
\draw[anchor = south] (f3) node {$y$};
\draw[anchor = south] (e3) node {$x$};

\end{tikzpicture}
\end{center}
\caption{Graphs house $H$, hole $H$, domino $D$ and $A$-graph (from left to right).} \label{hhd}
\end{figure}

\begin{lemma}\label{polb1}
Let $G$ be a graph. The $m^3$-transit function on $G$ satisfies $(b1)$ if and only if $G$ is $(HHD)$-free.
\end{lemma}

\begin{proof}
If $G$ contains a house, a hole or a domino as an induced subgraph, then $m^3$-transit function does not satisfy  $(b1)$ for vertices denoted on Figure~\ref{hhd}. Conversely, suppose that the $m^3$-transit function does not satisfy $(b1)$ on $G$. That is, $x\in m^3(u,v)$, $v\neq x$ and $v\in m^3(u,x)$. Indeed, $x\in m^3(u,v)$ implies that there is an $u,v$-induced path of length at least three containing $x$, and $v\in m^3(u,x)$ implies that there is an  $u,x$-induced path of length at least three containing $v$. Let $P=ux_1 \dots x_nx$ be a $u,x$-induced path, $Q=xv_1\dots v_tv$ be an $x,v$-induced path and $R:uu_1\dots u_mv$ be a $u,v$-induced path without a neighbor of $x$ (except possibly $v$) where $Q\cup R$ form an induced $x,u$-path. We have $n,m\geq 1$ since $ux,vu\notin E(G)$, while $t\geq 0$ ($t=0$ means that $xv\in E(G)$). Let $a$ be the vertex common to $P$ and $R$ closest to $x$. Clearly, $a$ is not $x$ and not a neighbor of $x$ on $P$. If $a\neq u$, then replace $a$ by $u$ without changing the assumptions. In addition, let us choose $u,x,v$ in such a way that $P$ is a short as possible. If $u_1x_i\in E(G)$ for some $i\in \{1,\dots,n-1\}$, then we can replace $u$ by $x_i$ and we have a contradiction with the choice of $P$. This means, that $x_1,\dots,x_{n-1}$ are not adjacent to $u_1$. If $n>2$, then $u_1,u,x_1,\dots,x_n$ (possibly together with some other vertices) form an induced cycle of length at least five and we have a hole. 

Let now $n=2$. If $x_2u_1\notin E(G)$, then $u_1,u,x_1,x_2,x$ (possibly together with some other vertices) form a hole again. So, we may assume that $x_2u_1\in E(G)$. If $m=1$, then we have a domino if $Q=xv$ and a hole if $t\geq 1$. Now we have $m>1$. If $x_2u_2\in E(G)$, then $\{u,u_1,u_2,x_2,x_1\}$ induce a house. Otherwise, when $x_2u_2\notin E(G)$, then vertices $\{u,u_1,u_2,u_3,x_2,x_1\}$ induce a domino when $x_2u_3\in E(G)$ (notice that this is possible when $m\geq 3$) or $\{x_2,u_1,\dots,u_i\}$ a hole when $u_i$, $i\geq 3$, is the first vertex adjacent to $x_2$ on $R$ after $u_1$. Finally, $x_2u_1\xrightarrow{R}v\xrightarrow{Q}xx_2$ is a hole when $x_2$ is not adjacent to any other vertex of $R$ than $u_1$.

We are left with $n=1$. If $x_1$ is not adjacent to two consecutive vertices $u_i$ and $u_{i+1}$ (where $u_{i+1}$ may equal to $v$), then $x_1,u_{i-1},u_i,u_{i+1}u_{i+2}$ with possibly some other vertices induce a hole (if $u_{i+1}=v$, then $u_{i+2}=v_t$). Suppose next that $x_1u_i\notin E(G)$ for some $i\in \{1,\dots,m\}$. This means that $u_{i-1}x_1,u_{i+1}x_1\in E(G)$, where $u_{i-1}$ may equal to $u$ and $u_{i+1}$ may equal to $v$. Let $y\in \{u_{i-2},u_{i+2}\}$ when it exists. If $yx_i\in E(G)$, then $x_1,u_{i-1},u_i,u_{i+1},y$ induce a house. Otherwise, if $yx_i\notin E(G)$, then $x_1,u_{i-1},u_i,u_{i+1},y$ together with the other neighbor $z$ of $y$ on $R$ induce a domino. Here $z$ may equal to $u$ when $y=u_1$ or $z$ may equal to $v_t$ when $y=v$ and $x_1v_t\in E(G)$ in this case, since otherwise $x_1,u_{m},v,v_t,v_{t-1}$ and possibly some other vertices induce a hole. The last option is that $y$ does not exists, which means that $m=1$ and $x_1u_1\in E(G)$, because otherwise two consecutive vertices $u_1$ and $v$ of $R$ are not adjacent to $x_1$. If $xv\in E(G)$, then $u,x_1,x,v,u_1$ induce a house. Otherwise, let $v_j$ be the last neighbor of $x_1$ on $Q$ for $j\in \{0,1,\dots,t\}$, where $v_0=x$. If $j=t$, then $u,u_1,v,v_t,x_1$ induce a house. If $j<t$, then $v_j,v_{j+1},\dots,v_t,v,u_1,x_1$ induce a hole. Thus we obtained a hole, a house or a domino in every possible situation and the proof is completed.  
\end{proof} 

\begin{lemma}\label{poljo}
   Let $G$ be a graph. The $m^3$-transit function defined on $V(G)$ satisfies $(J0)$ if and only if $G$ is (hole,$A$)-free.
\end{lemma}
\begin{proof}
First assume $G$ contains a hole graph with $y,v,u,x$ as its consecutive vertices or a $A$-graph with vertices as shown on Figure~\ref{hhd}, then  $x\in m^3(u,y)$, $y\in m^3(x,v)$  and $x\notin m^3(u,v)$. That is, the $m^3$-transit function does not satisfy $(J0)$. 

Conversely, suppose that $(J0)$ does not hold on vertices $u,v,x,y$ for $m^3$-transit function on $G$. There exists an induced $u,y$-path, say $P$, containing $x$ of length at least three and an induced $x,v$-path, say $Q$, containing $y$ of length at least three. Since $x\notin m^3(u,v)$ there is a chord $ab$ from a vertex in the $u,x$-subpath of $P$ to a vertex in  the $y,v$-subpath of $Q$. Let $a$ be the vertex closest to $x$ on $P$ that has a neighbor $b$ on $Q$. In this case, the chord $ab$ and the path  $a\xrightarrow{P}y\xrightarrow{Q}b$ form a cycle $C$. If $C$ is not induced, then let $a'$ is the first vertex from $x,y$-subpath of $P$ with a neighbor $b'$ from $y,b$-subpath of $Q$. Further we choose $b'$ to be first such vertex on $y,b$-subpath of $Q$. Clearly, $a\neq x\neq a'$. Cycle $a\xrightarrow{P}a'b'\xrightarrow{Q}ba$ is a hole and we are done. Otherwise $C$ is induced of lengrt at least four. If the length of $C$ is at least five, then we have a hole again. If $C=axyba$, then $a\neq u$ and $b\neq v$ because $P$ and $Q$ have lenght at least three. Let $u'$ be the neighbor of $a$ in the $a,u$-subpath $P$ and let $v'$ be the neighbor of $b$ in the $b,v$-subpath $Q$. Vertices $u'$ and $v'$ together with the vertices of $C$ induce an $A$-graph. 
\end{proof}
 
\begin{lemma}\label{polCg}
The  $m^3$-transit function satisfies $(Ch)$ on $HHDA$-free graphs.
\end{lemma}
\begin{proof}
Suppose that $m^3$-transit function does not satisfy $(Ch)$ on the weak bipolarizable graphs. That is $x\in m^3(u,v)$, $y\in m^3(x,w)$,  $y\notin m^3(u,v)$, $y\notin m^3(u,w)$ and $y\notin m^3(v,w)$. Let $P= u_nu_{n-1}\dots$ $u_1xv_1\dots v_m$, $u=u_n$ and $v=v_m$, is an induced $u,v$-path of length at least three containing $x$ and $Q=x_px_{p-1}\dots x_1yw_1\dots w_q$, $x_p=x$ and $w_q=w$ be an induced $x,w$-path of length at least three containing $y$. Since $y\notin m^3(u,v)$, $y$ is not on $P$. 
     
First assume that $yw\in E(G)$, which means that $yx\notin E(G)$. Condition $y\notin m^3(u,w)$ implies that $u$ is adjacent to $y$ or to $w$ and  from $y\notin m^3(v,w)$ it  follows that $v$ is adjacent to $y$ or $w$. If both $u$ and $v$ are adjacent to the same vertex, let it be $y$ (the other case is symmetric), then $P$ together with the edges $uy$ and $vy$ form a cycle $C$ of length at least five. The only possible chords in $C$ are between $y$ and vertices of $P$ different than $x$. If $u_1y\notin E(G)$, then we have a hole on $u_2,u_1,x,v_1$ and $y$ possibly with some other vertices of $P$. Similarly, we obtain a hole when or $v_1y\notin E(G)$. Let now $u_1y,v_1y\in E(G)$, which yields a four-cycle $u_1xv_1yu_1$. If $u_2y\in E(G)$ (resp. $v_2y\in E(G)$), then we have a house on $u_2,u_1,x,v_1,y$ (resp. $v_2,v_1,x,u_1,y$). So, let $u_2y,v_2y\notin E(G)$. If $n\geq 4$ (resp. $m\geq 4$), then we have an induced $A$ on $u_2,u_1,x,v_1,y,u_4$ (resp. $v_2,v_1,x,u_1,y,v_4$) and we may assume that $n,m\leq 3$. Also $n\neq 2$ and $m\neq 2$, because $u_2y,v_2y\notin E(G)$. So, $n,m\in\{1,3\}$ and at least one of them equals three since $P$ is of length at least three. Finally, $n=3$ (resp. $m=3$) gives an induced domino on $u_2,u_1,x,v_1,y,u_3$ (resp. $v_2,v_1,x,u_1,y,v_3$). Now consider the case when $u$ is adjacent to $y$ and $v$ is adjacent to $w$ (or vise versa), which means also that $uw,vy\notin E(G)$. As before we have a hole or an induced four cycle $C=u_1xv_1yu_1$. In addition, edge $u_2y$ yields a house, $u_2y\notin E(G)$ and $u_3y\in E(G)$ induces a domino and $u_2y,u_3y\notin E(G)$ gives $A$ when $n\geq 4$. So, let $n=1$, which means that $m\geq 2$ and $u_1w\notin E(G)$. But now we have a house on $u_1,x,v_1,w,y$ when $v_1w\in E(G)$ or on $u_1,x,v_1,v_2,y$ when $v_2y\in E(G)$. Otherwise, $v_1w,v_2y\notin E(G)$ and $u_1,x,v_1,v_2,y,w$ induce a domino if $v_2w\in E(G)$ or $A$-graph when $v_2w\in E(G)$.
     
Now assume that $yw\notin E(G)$. Conditions $y\notin m^3(u,w)$   and  $y\notin m^3(v,w)$ implies that there is a chord from $u,x$-subpath of $P$ to $y,w$-subpath of $Q$ and a chord from $v,x$-subpath of $P$ to $y,w$-subpath of $Q$. Let $a$ be the vertex closest to $x$ in the $u,x$-subpath of $P$, which has a neighbor $a'$ closest to $y$ in the $y,w$-subpath of $Q$. Similarly, let $b$ be the vertex closest to $x$ in the $v,x$-subpath of $P$, which has a neighbor $b'$ closest to $y$ in the $y,w$-subpath of $Q$. Clearly, both $a'$ and $b'$ are different from $y$. If $a\neq u_1$, then $a\xrightarrow{P}x\xrightarrow{Q}a'a$ is a hole. So, we may assume that $a=u_1$ and by symmetric argument also  $b=v_1$. Let first $a'\neq b'$ and we may assume that $b'$ is closer to $y$ than $a'$. Again $ax\xrightarrow{Q}a'a$ is a hole. Hence, let $a'=b'$. Furthermore, if one of $xy$ and $ya'$ is not an edge, we again have a hole $ax\xrightarrow{Q}a'a$. So, $Q=xya'w_2\dots w_q$ where $w_q=w$ and $q\geq 2$. At least one of $n$ and $m$ is at least two, say $n\geq 2$, since $P$ is of length at least three. Notice that $C=axba'a$ is an induced cycle. If $u_2a'\in E(G)$, then we have an induced house on $C$ together with $u_2$. Hence, $u_2a'\notin E(G)$. Further, $aw_2\notin E(G)$ yields an induced $A$ on $C$ together with $u_2$ and $w_2$ when also $u_2w_2\notin E(G)$ and an induced domino on the same vertices when $u_2w_2\in E(G)$. So, let $aw_2\in E(G)$. If $bw_2\notin E(G)$, then we have an induced house on $C$ together with $w_2$. Thus, we may assume that $bw_2\in E(G)$. But now $b,x,y,a',w_2$ induce a house if $by\notin E(G)$ and we may assume that $by\in E(G)$. Further, $b,x,y,a,w_2$ induce a house if $ay\notin E(G)$ and we may assume that $ay\in E(G)$. If $u_2y\in E(G)$, then $b,y,a,w_2,u_2$ induce a house. So, $u_2y\notin E(G)$ and $u_2\neq u$, because otherwise there exists an induce $u,v$-path of length at least three that starts $u_2ay$ and continues with the last neighbor of $y$ on $x,v$-subpath of $P$, a contradiction with $y\notin m^3(u,v)$. If $u_3y\in E(G)$, then we have an induced house on $u_3,u_2,a,y,x$. So, $u_3y\notin E(G)$ and $u_3\neq u$, because otherwise there exists an induce $u,v$-path of length at least three, similar as before. Now, first edge $u_iy$ for $i\geq 4$ induce a hole on $y,a,u_2,\dots,u_i$. If such an edge does not exists, then we have an induced $u,v$-path of length at least three that contains $y$, a contradiction because $y\notin m^3(u,v)$. In all cases we have a contradiction to the assumption that $G$ is a weak bipolarizable graph and the proof is complete.   
\end{proof}

The following proposition follows from Lemmas~\ref{polb1}, \ref{poljo} and \ref{polCg} and it serves as a characterization of the $m^3$-transit function of weak bipolarizable graphs. However, it is clear from Lemmas~\ref{polb1} and \ref{poljo} that $(b1)$ and $(J0)$ alone are sufficient for characterizing the $m^3$-transit function of weak bipolarizable graphs.

\begin{proposition}
Let $G$ be a graph. The $m^3$-transit function defined on  $V(G)$ satisfies $(b1)$, $(J0)$ and $(Ch)$ if and only if $G$ is a weak bipolarizable graph.
\end{proposition}

\subsection{All-paths convexity}

The coarsest path transit function of a graph $G$ is the \emph{all-paths transit function}, represented as $A(u, v) = \{w \in V(G) : w$ lies on some $u, v$-path$\}$, containing all vertices that exist on any $u, v$-path. For any two vertices $u$ and $v$ in a connected graph $G$, it is evident that $I(u,v)\subseteq J(u,v)\subseteq A(u,v)$. The characterization of the all-paths transit function is provided in \cite{chklmu-01}, and it is established that $A$ satisfies $(b1)$ if and only if $G$ is a tree. Our objective is to identify all graphs in which $A$ satisfies $(J0)$ as well as $(Ch)$.

\begin{lemma}
Let $G$ be a graph. The all-paths transit function $A$ satisfies $(J0)$ on $G$ if and only if every component of $G$ is a tree or a two-connected graph. 
\end{lemma} 
\begin{proof}
Suppose first that there exists a component $G_1$ of $G$ that is not a tree and is not two-connected. So, there exists a cut-vertex $v$ on a cycle $C$. Let $x$ and $y$ be different vertices of $C$ that are different from $v$ and let $u$ be a neighbor of $v$ from a different component of $G_1-v$ than $C$. Now, $x$ belongs to $A(u,y)$, $y$ belongs to $A(x,v)$, but $x$ does not belong to $A(u,v)$. In other words, $A$ does not satisfy $(J0)$. 

Conversely, suppose that $A$ does not satisfy $(J0)$ in a component $G_1$ of $G$. That is, $x\in A(u,y)$, $y\in A(x,v)$ but $x\notin A(u,v)$. That is, there is a $u,y$-path, say $P$, containing $x$ and $x,v$-path, say $Q$, containing $y$, but there is no $u,v$-path containing $x$. This implies that $Q$ passes through a vertex $w$ in the $u,x$-subpath of $P$ and $w$ separates $x$ and $u$. Thus, $w$ is a cut-vertex and $G_1$ is not two-connected. Also, $P\cup Q$ induced at least one cycle and $G_1$ is not a tree. 
\end{proof}

Furthermore, notice that all-paths transit function satisfy $(Ch)$ for every graph. Thus, according to Theorem ~\ref{cg}, we obtain the following.

\begin{theorem}
The all-path convexity of a graph $G$ forms a convex geometry if and only if $G$ is a tree.
\end{theorem}

\subsection{Toll convexity}\label{tc}

A \emph{toll walk} between two different vertices $w_1$ and $w_k$ of a finite connected graph $G$ are vertices $w_{1},\dots ,w_{k}$ that satisfy the following conditions:
\begin{itemize}
	\item $w_{i}w_{i+1} \in E(G)$ for every $i\in\{1,\dots,k-1\}$,
	\item $ w_1w_{i} \in  E(G)$ if and only if $ i=2$,
	\item $ w_kw_{i} \in  E(G)$ if and only if $ i=k-1$.
\end{itemize}
That is, a toll walk $W$ from $u$ to $v\neq u$ is a walk in which $u$ is adjacent only to the second vertex of $W$ and $v$ is adjacent only to the for-last vertex of $W$. 

In order to characterize the dominating pairs of vertices in interval graphs, Alcon \cite{toll2}  introduced toll walks. Later, Alcon et al. \cite{toll1} defined the toll interval $T(u,v)$ as the set of all vertices that belong to the toll walks between $u$ and $v\neq u$. This gives rise to the toll walk transit function $T:V(G)\times V(G)\rightarrow 2^{V(G)}$ of a graph $G$ defined as $T(u,u)=\{u\}$ or, for $u\neq v$,   
$$ T_{G}\left( u,v\right) =\{x\in V( G ):x \text{ lies on a toll walk between } u \text{ and } v \}.$$  

In \cite{toll1}, Alcon et al. proved that interval graphs are convex geometries for the toll convexity as given below. \emph{Interval graphs} are the intersection graphs of the intervals on the real line. This means that a vertex of an interval graph is represented by one interval on the real line, and two intervals are adjacent whenever they intersect. In \cite{lekbo}, interval graphs are characterized as chordal, AT-free graphs. A set of three vertices in a graph $G$ such that each pair is joined by a path that avoids the neighborhood of the third vertex is known as an \emph{asteroidal triple} in $G$. A graph $G$ is called an \emph{AT-free graph} if it has no asteroidal triples.

\begin{theorem}\cite{toll1}\label{nas}
The toll convexity of a graph $G$ is a convex geometry if and only if $G$ is an interval graph.
\end{theorem}

The characterization of the toll walk transit function of the interval graphs is derived in \cite{lcp}, and it is presented below.

\begin{theorem}\cite{lcp}\label{preint}
The toll walk transit function $T$ on a graph $G$ satisfies (b1) and (J0) if and only if $G$ is an interval graph.
\end{theorem}

Next, we check the status of $(Ch)$ on interval graphs to show that Theorems \ref{cg} and \ref{nas} coexist.

\begin{proposition}\label{toll}
The toll walk transit function $T$ satisfies $(Ch)$ on interval graphs.
\end{proposition}

\begin{proof}
Suppose that the toll walk transit function $T$ does not satisfy $(Ch)$ on an interval graph $G$. Since $x\in T(u,v)$ and $y\in T(x,w)$, there is an induced $u,x$-path, say $P$, without a neighbor of $v$ (where $x$ is a possible exception), there is an induced $x,v$-path, say $Q$, without a neighbor of $u$ (where $x$ is a possible exception), there is an induced $x,y$-path, say $R$, without a neighbor of $w$ (where $y$ is a possible exception) and there is an induced $y,w$-path, say $S$, without neighbor of $x$ (where $y$ is a possible exception). Also, since $y\notin T(u,v)$ we may assume that $N(v)$ separates $u$ from $y$. If $ux,xv\in E(G)$, then there exists a neighbor $u'$ of $u$ and a neighbor $v'$ of $v$ on  $S$, because $y\notin T(u,w)$ and $y\notin T(v,w)$, respectively. If we choose $u'$ and $v'$ as close as possible on $S$, then $u,x,v,u',v'$ induce a $C_4$ if $u'=v'$ or $C_n, n>4$ if $u' \neq v'$ and $G$ is not chordal. If $ux \notin E(G)$ and $xv \in E(G)$, then $v$ has a neighbor $v'$ on $S$ as before. Moreover, $u$ has a neighbor $u'$ on $S$ or $w$ has a neighbor $w'$ in $P$ because $y\notin T(u,w)$. Let $a$ be the vertex closest to $x$ on $P$ having a neighbor $a'$ on $S$ (notice that $a$ cannot be $x$, while $a'$ can be $w$). If we choose $v'$ and $a'$ close as possible, then vertices of the sequence $a\xrightarrow{P}xvv'\xrightarrow{S}a'$ induces a cycle of length at least four and $G$ is not chordal. We can apply symmetric arguments when $ux\in E(G)$ and $xv\notin E(G)$. Finally, suppose that $ux\notin E(G)$ and $xv \notin E(G)$. If $xy\in E(G)$, then $u\xrightarrow{P}xyx\xrightarrow{Q}v$ is a $u,v$-toll walk containing $y$, that is $y\in T(u,v)$ which is not possible. Thus, $xy\notin E(G)$. As before, let $a$ be the vertex closest to $x$ in $P$ having a neighbor $a'$ on $S$, which exists since $y\notin T(u,w)$. Since $N(v)$ separates $u$ from $y$, $v$ has a neighbor $v'$ on $R$ and a neighbor $v''$ on $S$ and no vertex of the path $v'\xrightarrow{R}y\xrightarrow{S}v''$ is adjacent to $u$. Now, $vv'\xrightarrow{R}y\xrightarrow{S}v''v$ or $a\xrightarrow{P}x\xrightarrow{Q}vv''\xrightarrow{S}a'a$ contains an induced cycle of length at least four or the vertices $u,x,v$ form an asteroidal triple. 
\end{proof}

An example underline the last sentence of the above proof is on Figure 5 of \cite{toll1}. To translate it into our notation mark that $a=u$, $b=v$, $c=w$, $y$ is fine and our $x$ is the middle vertex in the lower line. Graph from this figure is chordal, but $u,v,x$ form an asteroidal triple.    

\subsection{Weak toll convexity}\label{wtc}

Dourado in \cite{wtwf} introduced the concept of \emph{weak toll walk} which is defined as a walk $W: uw_1\dots w_{k-1}v$ between vertices $u$ and $v$, $u\neq v$, such that $u$ is adjacent only to $w_1$ on $W$, which may appear multiple times in the walk, and $v$ is adjacent only to $w_{k-1}$ on $W$, which can also appear multiple times on $W$. In more detail, a weak toll $u,v$-walk, $u\neq v$, in $G$ is a sequence of vertices of the form
$uw_1\dots w_kv$, where the following conditions are satisfied:
\begin{itemize}
	\item $w_{i}w_{i+1} \in E(G)$ for every $i\in\{1,\dots,k-1\}$,
	\item $ uw_{i} \in  E(G)$ implies $w_i = w_1$,
	\item $ ww_{i} \in  E(G)$ implies $w_i = w_k$.
\end{itemize}
The \emph{weak toll interval} between $u$ and $v$, $u\neq v$, in $G$ is  
$$ W_T\left( u,v\right) =\{x\in V( G ):x \text{ lies on a weak toll walk between } u \text{ and } v \}.$$
If we add $W_T(u,u)=\{u\}$ for every $u\in V(G)$, then $W_T$ becomes a transit function.

The proper interval graphs are considered to be the convex geometries with respect to the weak toll convexity, as discussed in \cite{wtwf}. \emph{Proper interval graphs} are interval graphs that have an interval representation in which each interval is of unit length. Also, the proper interval graphs are exactly the claw-free interval graphs.

\begin{theorem}\cite{wtwf}
The weak toll convexity of a graph $G$ is a convex geometry if and only if $G$ is a proper interval graph.
\end{theorem}

The following proposition (Proposition 1 from \cite{lcj}) establishes that the weak toll walk transit function $W_T$ and the toll walk transit function $T$ coincide in claw-free graphs, including proper interval graphs as a particular case. 

\begin{proposition}\label{prop1}\cite{lcj}
 A graph $G$ is a claw-free graph if and only if $T(u,v)=W_T(u,v)$ for every $u,v\in V(G)$.
\end{proposition}

Now, if $G$ contains a claw with $u,x,v$ as its pendant vertices, then $x\in W_T(u,v)$, $x\neq v$, and  $v\in W_T(x,u)$, so that $(b1)$ is not satisfied by $W_T$. So, if $W_T$ satisfies $(b1)$, then $G$ is claw-free. Consequently, by Proposition \ref{prop1}, we have $T=W_T$. Now, by Theorem \ref{preint}, if $G$ is claw-free and $W_T$ satisfies $(b1)$ and $(J0)$, then $G$ is a proper interval graph and conversely, if $G$ is a proper interval graph, then $W_T$ satisfies $(b1)$ and $(J0)$. These arguments lead to the following theorem; See, also \cite{lcj}, for the detailed proof of the Theorem~\ref{Weak-Toll}. 

\begin{theorem}\cite{lcj}\label{Weak-Toll}
The weak toll walk transit function $W_T$ of a graph $G$ satisfies $(b1)$ and $(J0)$ if and only if $G$ is a proper interval graph.
\end{theorem}

Also, since $T$ satisfies $(Ch)$ on interval graphs by Proposition \ref{toll}, Proposition \ref{prop1} implies that $W_T$ satisfies $(Ch)$ on proper interval graphs.

\subsection{$P_3$-convexity}

A path of length two in a graph $G$ (not necessarily induced) is called a $P_3$-path. A $P_3$-transit function is defined  by $P_3: V\times V \longrightarrow 2^{V}$ where
$$P_3(u,v)=\{w\in V(G): w \text{ lies on some }u,v\text{-}P_3\text{-path  in }G \}$$ 
In \cite{wtwf1},  it was proved that the forest of stars are the convex geometry with respect to $P_3$-convexity.  

\begin{theorem}\cite{wtwf1}\label{P3conv}
The $P_3$-convexity of a graph $G$ is a convex geometry if and only if $G$ is a forest of stars.
\end{theorem}

\begin{lemma}\label{stb1}
Let $G$ be a graph. The $P_3$-transit function on $G$ satisfies $(b1)$ if and only if $G$ is triangle-free.
\end{lemma}
\begin{proof}
If $G$ contains a triangle with vertices $u,v,x$, then $x\in P_3(u,v)$, $v\neq x$ and $v\in P_3(u,x)$ so that $P_3$-transit function does not satisfy $(b1)$. Conversely, suppose that $P_3$-transit function does not satisfies $(b1)$. That is $x\in P_3(u,v)$, $v\neq x$ and $v\in P_3(u,x)$. Now $x\in P_3(u,v)$ and $v\in P_3(u,x)$ implies that $ux,xv,uv\in E(G)$ and $u,v,x$ induce a triangle. 
\end{proof}

Next, we characterize the graphs in which the $P_3$-transit function satisfies $(J0)$. The graphs are precisely $(P_4,C_4,K_4^-,3$-pan$)$-free, where $K_4^-$ is a complete graph on four vertices without one edge and $3$-pan is a triangle with an additional pendant edge.

\begin{lemma}\label{stjo}
Let $G$ be a graph. The $P_3$-transit function defined on $V(G)$ satisfies  $(J0)$ if and only if $G$ is $(P_4,C_4,K_4^-,3$-pan$)$-free.
\end{lemma}

\begin{proof}
First, assume that $G$ contains either a $P_4$ or a cycle $C_4$ as an induced subgraph with $u,x,y,v$ as its consecutive vertices. In both cases, it holds that $x\in P_3(u, y)$, $y\in P_3(x, v)$ and $x\notin P_3(u, v)$ and $P_3$-transit function does not satisfy $(J0)$. For $K_4^-$ let $u$ and $y$ be vertices of degree three and $x$ and $v$ vertices of degree two. It is easy to check that $(J0)$ does not hold for $u,v,x,y$. Now, assume that $G$ contains a $3$-pan with $u,x,y$ forming the triangle and $v$ being the pendent vertex adjacent to $y$. Also in this case the $P_3$-transit function does not satisfy $(J0)$.
   
Conversely, suppose that $P_3$ does not satisfy $(J0)$ for some vertices $u,x,y,v$. Now, $x\in P_3(u,y)$ implies $ux,xy\in E(G)$ and $y\in P_3(x,v)$ yields $xy,yv\in E(G)$. Moreover, $xv\notin E(G)$ because $x\notin P_3(u,v)$. Vertices $u,x,y,v$ induce $K_4^-$ if $uy,uv\in E(G)$, or $3$-paw if $uy\in E(G)$ and $uv\notin E(G)$, or $C_4$ if $uy\notin E(G)$ and $uv\in E(G)$, or finally $P_4$ when $uy,uv\notin E(G)$. 
\end{proof}
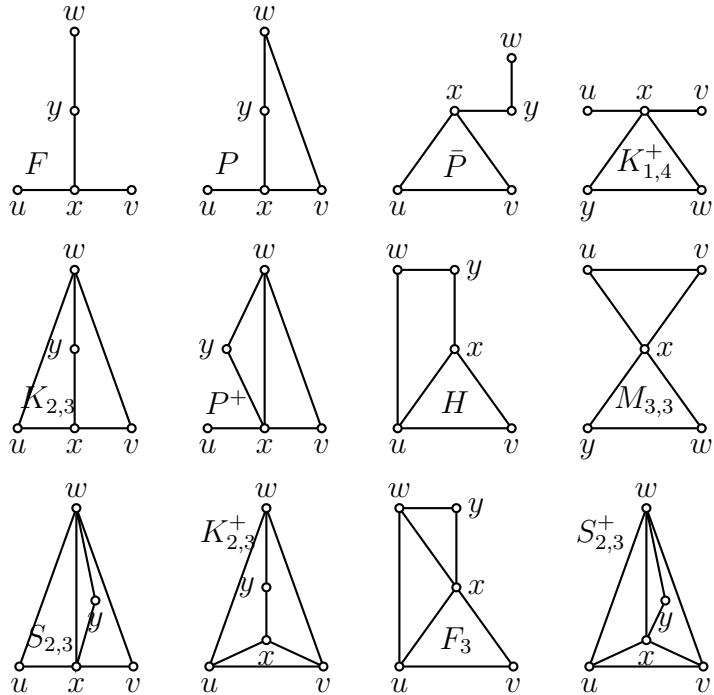
\begin{figure}[ht]
\begin{center}
\begin{tikzpicture}[scale=0.5,style=thick,x=1cm,y=0.7cm]
\def\vr{3pt} % \vr = vertex radius;

% define vertices
%%%%% %%%%% $K_{1,4}^+$
\path (15,0) coordinate (a);
\path (18,0) coordinate (b);
\path (16.5,3) coordinate (c);
\path (15,3) coordinate (d);
\path (18,3) coordinate (e);

%  edges
\draw (a) -- (b) -- (c)  -- (a);
\draw (c) -- (e) -- (d);

\draw (a) [fill=white] circle (\vr);
\draw (b) [fill=white] circle (\vr);
\draw (c) [fill=white] circle (\vr);
\draw (d) [fill=white] circle (\vr);
\draw (e) [fill=white] circle (\vr);

\draw[anchor = north] (a) node {$y$};
\draw[anchor = north] (b) node {$w$};
\draw[anchor = south] (e) node {$v$};
\draw[anchor = south] (c) node {$x$};
\draw[anchor = south] (d) node {$u$};
\draw(16.5,1) node {$K_{1,4}^+$};
%%%%%%%%%%%%%%%%%%%%\bar{P}

\path (10,0) coordinate (a1);
\path (13,0) coordinate (b1);
\path (11.5,3) coordinate (c1);
\path (13,3) coordinate (d1);
\path (13,5) coordinate (e1);

%  edges
\draw (a1) -- (b1)--(c1)--(a1); 
\draw (c1) -- (d1) -- (e1) ;

\draw (a1) [fill=white] circle (\vr);
\draw (b1) [fill=white] circle (\vr);
\draw (c1) [fill=white] circle (\vr);
\draw (d1) [fill=white] circle (\vr);
\draw (e1) [fill=white] circle (\vr);

\draw[anchor = north] (a1) node {$u$};
\draw[anchor = north] (b1) node {$v$};
\draw[anchor = south] (e1) node {$w$};
\draw[anchor = west] (d1) node {$y$};
\draw[anchor = south] (c1) node {$x$};
\draw(11.5,1) node {$\bar{P}$};
%%%%%%%%%%%%%%%%%%%%%%Fork

\path (0,0) coordinate (a2);
\path (3,0) coordinate (c2);
\path (1.5,0) coordinate (d2);
\path (1.5,3) coordinate (e2);
\path (1.5,6) coordinate (f2);

%  edges
\draw (a2) -- (d2)--(c2);
\draw (d2) -- (e2) -- (f2);

\draw (a2) [fill=white] circle (\vr);
\draw (c2) [fill=white] circle (\vr);
\draw (d2) [fill=white] circle (\vr);
\draw (e2) [fill=white] circle (\vr);
\draw (f2) [fill=white] circle (\vr);

\draw[anchor = north] (a2) node {$u$};
\draw[anchor = north] (c2) node {$v$};
\draw[anchor = north] (d2) node {$x$};
\draw[anchor = south] (f2) node {$w$};
\draw[anchor = east] (e2) node {$y$};
\draw(0.5,1) node {$F$};

%%%%%%%%%%%%%%%%%%%%%%P

\path (5,0) coordinate (a3);
\path (6.5,0) coordinate (b3);
\path (8,0) coordinate (c3);
\path (6.5,3) coordinate (d3);
\path (6.5,6) coordinate (e3);

%  edges
\draw  (a3) -- (b3) -- (c3);
\draw (b3) -- (d3) -- (e3)-- (c3) ;

\draw (a3) [fill=white] circle (\vr);
\draw (b3) [fill=white] circle (\vr);
\draw (c3) [fill=white] circle (\vr);
\draw (d3) [fill=white] circle (\vr);
\draw (e3) [fill=white] circle (\vr);

\draw[anchor = north] (a3) node {$u$};
\draw[anchor = north] (b3) node {$x$};
\draw[anchor = north] (c3) node {$v$};
\draw[anchor = east] (d3) node {$y$};
\draw[anchor = south] (e3) node {$w$};
\draw(5.5,1) node {$P$};
\end{tikzpicture}

\begin{tikzpicture}[scale=0.5,style=thick,x=1cm,y=0.7cm]
\def\vr{3pt} % \vr = vertex radius;

% define vertices
%%%%% %%%%% $M_{3,3}$
\path (15,0) coordinate (a);
\path (18,0) coordinate (b);
\path (16.5,3) coordinate (c);
\path (15,6) coordinate (d);
\path (18,6) coordinate (e);

%  edges
\draw (a) -- (b) -- (c)  -- (a);
\draw (c) -- (e) -- (d) -- (c);

\draw (a) [fill=white] circle (\vr);
\draw (b) [fill=white] circle (\vr);
\draw (c) [fill=white] circle (\vr);
\draw (d) [fill=white] circle (\vr);
\draw (e) [fill=white] circle (\vr);

\draw[anchor = north] (a) node {$y$};
\draw[anchor = north] (b) node {$w$};
\draw[anchor = south] (e) node {$v$};
\draw[anchor = west] (c) node {$x$};
\draw[anchor = south] (d) node {$u$};
\draw(16.5,1) node {$M_{3,3}$};
%%%%%%%%%%%%%%%%%%%%H

\path (10,0) coordinate (a1);
\path (13,0) coordinate (b1);
\path (11.5,3) coordinate (c1);
\path (11.5,6) coordinate (d1);
\path (10,6) coordinate (e1);

%  edges
\draw (a1) -- (b1)--(c1)--(a1); 
\draw (c1) -- (d1) -- (e1) -- (a1);

\draw (a1) [fill=white] circle (\vr);
\draw (b1) [fill=white] circle (\vr);
\draw (c1) [fill=white] circle (\vr);
\draw (d1) [fill=white] circle (\vr);
\draw (e1) [fill=white] circle (\vr);

\draw[anchor = north] (a1) node {$u$};
\draw[anchor = north] (b1) node {$v$};
\draw[anchor = south] (e1) node {$w$};
\draw[anchor = west] (d1) node {$y$};
\draw[anchor = west] (c1) node {$x$};
\draw(11.5,1) node {$H$};
%%%%%%%%%%%%%%%%%%%%%%K_2,3

\path (0,0) coordinate (a2);
\path (3,0) coordinate (c2);
\path (1.5,0) coordinate (d2);
\path (1.5,3) coordinate (e2);
\path (1.5,6) coordinate (f2);

%  edges
\draw (a2) -- (d2)--(c2);
\draw (d2) -- (e2) -- (f2);
\draw (a2) -- (f2) -- (c2);

\draw (a2) [fill=white] circle (\vr);
\draw (c2) [fill=white] circle (\vr);
\draw (d2) [fill=white] circle (\vr);
\draw (e2) [fill=white] circle (\vr);
\draw (f2) [fill=white] circle (\vr);

\draw[anchor = north] (a2) node {$u$};
\draw[anchor = north] (c2) node {$v$};
\draw[anchor = north] (d2) node {$x$};
\draw[anchor = south] (f2) node {$w$};
\draw[anchor = east] (e2) node {$y$};
\draw(0.8,1) node {$K_{2,3}$};
 
%%%%%%%%%%%%%%%%%%%%%%P^+

\path (5,0) coordinate (a3);
\path (6.5,0) coordinate (b3);
\path (8,0) coordinate (c3);
\path (5.5,3) coordinate (d3);
\path (6.5,6) coordinate (e3);

%  edges
\draw  (a3) -- (b3) -- (c3);
\draw (b3) -- (d3) -- (e3)-- (c3) ;
\draw  (e3) -- (b3);

\draw (a3) [fill=white] circle (\vr);
\draw (b3) [fill=white] circle (\vr);
\draw (c3) [fill=white] circle (\vr);
\draw (d3) [fill=white] circle (\vr);
\draw (e3) [fill=white] circle (\vr);

\draw[anchor = north] (a3) node {$u$};
\draw[anchor = north] (b3) node {$x$};
\draw[anchor = north] (c3) node {$v$};
\draw[anchor = east] (d3) node {$y$};
\draw[anchor = south] (e3) node {$w$};
\draw(5.5,1) node {$P^+$};
\end{tikzpicture}

\begin{tikzpicture}[scale=0.5,style=thick,x=1cm,y=0.7cm]
\def\vr{3pt} % \vr = vertex radius;

% define vertices
%%%%% %%%%% $S_{2,3}^+$

\path (15,0) coordinate (a2);
\path (18,0) coordinate (c2);
\path (16.5,1) coordinate (d2);
\path (17,2.5) coordinate (e2);
\path (16.5,6) coordinate (f2);

%  edges
\draw (a2) -- (d2)--(c2);
\draw (f2) -- (d2) -- (e2) -- (f2);
\draw (c2) -- (a2) -- (f2) -- (c2);

\draw (a2) [fill=white] circle (\vr);
\draw (c2) [fill=white] circle (\vr);
\draw (d2) [fill=white] circle (\vr);
\draw (e2) [fill=white] circle (\vr);
\draw (f2) [fill=white] circle (\vr);

\draw[anchor = north] (a2) node {$u$};
\draw[anchor = north] (c2) node {$v$};
\draw[anchor = north] (d2) node {$x$};
\draw[anchor = south] (f2) node {$w$};
\draw[anchor = north] (e2) node {$y$};

\draw(15.3,5) node {$S_{2,3}^+$};
%%%%%%%%%%%%%%%%%%%%H^+

\path (10,0) coordinate (a1);
\path (13,0) coordinate (b1);
\path (11.5,3) coordinate (c1);
\path (11.5,6) coordinate (d1);
\path (10,6) coordinate (e1);

%  edges
\draw (a1) -- (b1)--(c1)--(a1); 
\draw (e1) -- (c1) -- (d1) -- (e1) -- (a1);

\draw (a1) [fill=white] circle (\vr);
\draw (b1) [fill=white] circle (\vr);
\draw (c1) [fill=white] circle (\vr);
\draw (d1) [fill=white] circle (\vr);
\draw (e1) [fill=white] circle (\vr);

\draw[anchor = north] (a1) node {$u$};
\draw[anchor = north] (b1) node {$v$};
\draw[anchor = south] (e1) node {$w$};
\draw[anchor = west] (d1) node {$y$};
\draw[anchor = west] (c1) node {$x$};
\draw(11.5,1) node {$F_3$};
%%%%%%%%%%%%%%%%%%%%%%S_2,3

\path (0,0) coordinate (a2);
\path (3,0) coordinate (c2);
\path (1.5,0) coordinate (d2);
\path (2,2.5) coordinate (e2);
\path (1.5,6) coordinate (f2);

%  edges
\draw (a2) -- (d2)--(c2);
\draw (f2) -- (d2) -- (e2) -- (f2);
\draw (a2) -- (f2) -- (c2);

\draw (a2) [fill=white] circle (\vr);
\draw (c2) [fill=white] circle (\vr);
\draw (d2) [fill=white] circle (\vr);
\draw (e2) [fill=white] circle (\vr);
\draw (f2) [fill=white] circle (\vr);

\draw[anchor = north] (a2) node {$u$};
\draw[anchor = north] (c2) node {$v$};
\draw[anchor = north] (d2) node {$x$};
\draw[anchor = south] (f2) node {$w$};
\draw[anchor = north] (e2) node {$y$};

\draw(0.8,1) node {$S_{2,3}$};
%%%%%%%%%%%%%%%%%%%%%%K_{2,3}^+

\path (5,0) coordinate (a2);
\path (8,0) coordinate (c2);
\path (6.5,1) coordinate (d2);
\path (6.5,3) coordinate (e2);
\path (6.5,6) coordinate (f2);

%  edges
\draw (a2) -- (d2)--(c2);
\draw (d2) -- (e2) -- (f2);
\draw (c2) -- (a2) -- (f2) -- (c2);

\draw (a2) [fill=white] circle (\vr);
\draw (c2) [fill=white] circle (\vr);
\draw (d2) [fill=white] circle (\vr);
\draw (e2) [fill=white] circle (\vr);
\draw (f2) [fill=white] circle (\vr);

\draw[anchor = north] (a2) node {$u$};
\draw[anchor = north] (c2) node {$v$};
\draw[anchor = north] (d2) node {$x$};
\draw[anchor = south] (f2) node {$w$};
\draw[anchor = east] (e2) node {$y$};
\draw(5.5,5) node {$K_{2,3}^+$};
 
\end{tikzpicture}

\end{center}
\caption{Family ${\cal A}$.} \label{kkp}
\end{figure}

Now we give a forbidden induced subgraph characterization of graphs that fulfill $(Ch)$. For this we define a family of graphs 
$${\cal A}=\{F,P,\bar{P},K_{1,4}^+,K_{2,3},P^+,H,M_{3,3},S_{2,3},K_{2,3}^+,F_3,S_{2,3}^+\},$$ 
which are depicted on Figure \ref{kkp}. We briefly justify their notation (except for well known graphs $K_{2,3}$ and house $H$): $F$ is a fork graph obtained from $K_{1,3}$ by one subdivision of one edge, $P$ is also a $4$-pan graph, $\bar{P}$ is obtained from $3$-pan graph by one subdivision of the pendant edge, $K_{1,4}^+$ is obtain from $K_{1,4}$ by adding one edge between two leaves, $P^+$ is obtained from $P$ by adding the diagonal edge in $C_4$ that starts in a vertex of degree three, $M_{3,3}$ are two three-cycles amalgamated in one vertex ($M$ stands for M\"obious), $S_{2,3}$ is a split graph on $K_2$ and three independent vertices together with all edges in between them, $K_{2,3}^+$ is $K_{2,3}$ together with one edge between two vertices of degree two, $F_3=P_4\vee K_1$ is a $3$-fan (where $\vee$ is a join) and $S_{2,3}^+$ is $S_{2,3}$ together with one edge between two of its independent vertices.  

\begin{lemma}\label{stCg}
The  $P_3$-transit function on a graph $G$ satisfies $(Ch)$ if and only if $G$ is ${\cal A}$-free graphs.
 \end{lemma}
 
 \begin{proof}
If $G$ contain any  of the graphs from $\cal A$ as an induced subgraph, then $x\in P_3(u,v)$, $y\in P_3(x,w)$, $y\notin P_3(u,w)$, $y\notin P_3(u,v)$ and $y\notin P_3(v,w)$ and $(Ch)$ does not hold if we label the vertices as on Figure~\ref{kkp}. Thus, $P_3$-transit function does not satisfies $(Ch)$ on $G$. 

Conversely, assume that $P_3$-transit function does not satisfy $(Ch)$. That is $x\in P_3(u,v)$, $y\in P_3(x,w)$, $y\notin P_3(u,w)$, $y\notin P_3(u,v)$ and $y\notin P_3(v,w)$. This means that $ux,xv,xy,yw\in E(G)$ and $uy,vy\notin E(G)$. Based on (non)existence of edges $uw,vw,xw,uv$ we obtain the following induced subgraphs on $B=\{u,v,x,y,w\}$.  
\begin{itemize}
\item If $uw,vw,xw,uv\notin E(G)$, then $F$ is an induced subgraph on $B$.
\item If $uw\in E(G)$ and $vw,xw,uv\notin E(G)$ or $vw\in E(G)$ and $uw,xw,uv\notin E(G)$, then $P$ is an induced subgraph on $B$.
\item If $xw\in E(G)$ and $uw,vw,uv\notin E(G)$, then $K_{1,4}^+$ is an induced subgraph on $B$.
\item If $uv\in E(G)$ and $uw,vw,xw\notin E(G)$, then $\bar{P}$ is an induced subgraph on $B$.
\item If $uw,xw\in E(G)$ and $vw,uv\notin E(G)$ or $vw,xw\in E(G)$ and $uw,uv\notin E(G)$, then $P^+$ is an induced subgraph on $B$.
\item If $uw,uv\in E(G)$ and $vw,xw\notin E(G)$ or $vw,uv\in E(G)$ and $uw,xw\notin E(G)$, then $H$ is an induced subgraph on $B$.
\item If $xw,uv\in E(G)$ and $uw,vw\notin E(G)$, then $M_{3,3}$ is an induced subgraph on $B$.
\item If $uw,vw,xw\in E(G)$ and $uv\notin E(G)$, then $S_{2,3}$ is an induced subgraph on $B$.
\item If $uw,vw,uv\in E(G)$ and $xw\notin E(G)$, then $K_{2,3}^+$ is an induced subgraph on $B$.
\item If $uw,xw,uv\in E(G)$ and $vw\notin E(G)$ or $vw,xw,uv\in E(G)$ and $uw\notin E(G)$, then $F_3$ is an induced subgraph on $B$.
\item Finally, if $uw,vw,xw,uv\in E(G)$, then $S_{2,3}^+$ is an induced subgraph on $B$.
 \end{itemize} 
 So, in every case we obtain an induced subgraph from ${\cal A}$ in $G$.
 \end{proof}

 Following Lemmas \ref{stb1} and \ref{stjo} to detect all graphs for which the $P_3$-transit function fulfills both $(b1)$ and $(J0)$, we immediately see that this happens only when there are no cycles and every induced path is of length at most two. However, this yields a forest of stars. Since the forest of stars does not contain any graph from ${\cal A}$ as an induced subgraph, Theorem \ref{P3conv} now also follows from Theorem \ref{cg}. In the case of connected graphs, the following consequence follows directly.
 
\begin{corollary}
Let $G$ be a connected graph. The $P_3$-transit function defined on $V(G)$ satisfies $(Ch)$, $(b1)$ and $(J0)$ if and only if $G$ is a star graph.
\end{corollary}

\subsection{Cut-vertex transit function and Convex geometry of connected sets}\label{C-convexity}

In\cite{muld-08}, the cut-vertex transit function of a connected graph $G$ is defined to be the function $C:V\times V\to 2^V$ such that
\begin{equation}\label{c2}
C(u,v)=\{u,v\} \cup \{x\in V : x \;\text{is a cut-vertex between $u$ and $v$}\}.
\end{equation}

Following \cite{edel-jami}, a subset $K$ of the vertices of a graph $G$ is called \emph{connected} if the subgraph induced by $K$ is connected. The collection of connected sets on the vertices is called the \emph{connected alignment}. It is easy to see that connected alignment does not yield a convexity on all graphs. For instance in a four-cycle $v_1v_2v_3v_4v_1$ sets $\{v_1,v_2,v_3\}$ and $\{v_1,v_4,v_3\}$ belong to the connected alignment, but their intersection $\{v_1,v_3\}$ is not. This is not the case for \emph{block graphs}, which are graphs where every maximum component without a cut-vertex is a complete graph. Trees are an examples of block graphs. The connected alignment of $G$ is a convex geometry
if and only if $G$ is a connected block graph, see \cite{edel-jami} Theorem 3.7. Even more, we will show that the connected alignment $\mathcal{C}_G$ of a block graph $G$ coincides with the cut-vertex transit function $C$ of $G$, that is $\mathcal{C}_G=\mathcal{C}_C$.

\begin{lemma}
The connected alignment $\mathcal{C}_G$ of a block graph $G$ is a $C$-convexity, where $C$ is the cut-vertex transit function of $G$.
\end{lemma}

\begin{proof} 
Let $C$ be the cut-vertex transit function of $G$.  We prove $\mathcal{C}_G=\mathcal{C}_C$.
Let $U\in \mathcal{C}_G$. Let $u,v\in U$. Since $U$ is connected, $U$ contains a $u,v$-path. Since the cut-vertices that separate $u$ and $v$ belong to every $u,v$-path, $C(u,v)\subseteq U$, i.e. $U$ is a convex set of the $C$-convexity. Hence, $U\in \mathcal{C}_{C}$.
Now, let $U\in \mathcal{C}_{C}$. Let $u,v\in U$. If $u$ and $v$ are in the same block, then $uv$ is an edge. If they are in different blocks, then there exists a unique sequence of blocks between them that pairwise intersect at the cut vertices that separate $u$ and $v$. These cut vertices form the unique shortest $u,v$-path $P$ in $G$, see \cite[Prop.~4]{Changat:21}. Since $C(u,v)\subseteq U$, clearly, $V(P)$ is contained in $U$. Therefore, $U$ is connected and $U\in \mathcal{C}_G$. Thus, $\mathcal{C}_G=\mathcal{C}_{C}$. 
\end{proof}

According to Proposition~4 in \cite{Changat:21}, if $G$ is a block graph, then the cut-vertex transit function coincides with the interval function. Hence, we get the following.

\begin{corollary}\label{Connected_align}
If the connected alignment $\mathcal{C}_G$ of a connected graph $G$ is a convex geometry, then $\mathcal{C}_G=\mathcal{C}_I$ where $I$ is the interval function of $G$ and $G$ is a block graph. 
\end{corollary}

Since the interval function $I$ of a block graph satisfies $(Ch), (J0)$ and $(b1)$, the Corollary~\ref{Connected_align} also follows from Theorem \ref{cg}.

\section{Transit functions identifying convex geometries}\label{identified}

Any collection of subsets of a set $V$ is said to be a \emph{set system} on $V$. A set system of non-empty subsets of $V$ is called a \emph{clustering system} if it contains every singleton subset of $V$ and the base set $V$. These set systems are closely related to transit functions
(called ``Boolean dissimilarities'' in \cite{Barthelemy:08}) that in
addition satisfy the monotone axiom $(m)$.

A \emph{$\mathscr{T}$-system} is a set system on $V$ satisfying the following three axioms:

$(KS)$ For every $x\in V$ is $\{x\}\in\mathscr{C}$.

$(KR)$ For every $C\in\mathscr{C}$ there exist points $p,q\in C$ such that, if $p,q\in C'$, then $C\subseteq C'$ for every $C'\in \mathscr{C}$.

$(KC)$ For any different $p,q\in V$ we have
$\displaystyle \bigcap\{C\in\mathscr{C}|p,q\in C\} \in \mathscr{C}$.

%The set systems corresponding to monotone transit functions are related in 
Following \cite[Theorem 2.6]{Changat:19a}, there is a bijection between monotone transit functions $R$ and $\mathscr{T}$-systems on $V$ mediated by  
\begin{equation*}
	\begin{split}
		R_{\mathscr{C}}(x,y) &:= \bigcap\{ C\in\mathscr{C}:x,y\in C\} \\
  	\mathscr{C}_R        &:= \{ R(x,y):x,y\in V\}
	\end{split}
\end{equation*}
We say that a set system $\mathscr{C}$ \emph{is identified} by the transit
function $R$ if $\mathscr{C}=\mathscr{C}_R$ where $R=R_{\mathscr{C}}$. 
A set system is identified by a transit function if and only if it is a $\mathscr{T}$-system. Conversely, a transit function identifies a set system if and only if it is monotone. In this case, $R_{\mathscr{C}}$ is called the \emph{canonical transit function} of the set system $\mathscr{C}$ and $\mathscr{C}_R$ is called the \emph{system of transit sets} of $R$.

A \emph{$\mathscr{T}$-system} is a binary clustering system if it satisfies

$(K1)$ $V \in \mathscr{C}$.

As shown in \cite{Barthelemy:08,Changat:19a}, binary clustering systems are
identified by monotone transit functions satisfying the additional
condition

$(a')$ There exist $u,v\in V$ such that $R(u,v)=V$.

A set system $(V,\mathscr{C})$ is \emph{closed (under non-empty intersection)} if it satisfies

$(K2)$ If $A,B\in\mathscr{C}$ and $A\cap B\ne\emptyset$, then	$A\cap B\in\mathscr{C}$.

A set system $\mathscr{C}$ on a finite non-empty set $V$ satisfies (K1) and (K2) if and only if $\mathscr{C}\cup\emptyset$ is a convexity in the usual sense. %In the context of clustering systems, we refer to such a set system as a \emph{convexity}. 
Convexities satisfying (KS) are therefore the same as \emph{closed clustering systems}. It is shown in \cite{Changat:19a} that (K2) for $\mathscr{C}$ is
equivalent to the following property of the identifying transit function:

$(k)$ For all $u,v,x,y\in V$ with $R(u,v)\cap R(x,y)\ne\emptyset$, there are $p,q\in V$ such that $R(u,v)\cap R(x,y) = R(p,q)$.

For a monotone transit function $R$, $\mathscr{C}_R$ is a convexity if and only if $R$ satisfies $(a')$ and $(k)$.

In the following, we characterize transit functions identifying convex geometries. 
%We consider only monotone transit functions since we know that a transit function that identifies a set system must be monotone. 
Since the transit sets of a monotone transit function are $R$-convex sets, $\mathscr{C}_R\subseteq \mathcal{C}_R$.
%Clearly, the set of all transit sets of a monotone transit function is a subset of the $R$-convexity. 
However, Example \ref{ex:cgjob1-congeo} shows that a convex geometry $\mathcal{C}_R$ does not imply that the set of transit sets $\mathscr{C}_R$ also forms a convex geometry, even if the transit function satisfies $(m)$, $(Ch)$, $(J0)$, $(b1)$, $(a')$ and $(k)$.
 
\begin{example}\label{ex:cgjob1-congeo}
Let $V=\{a,b,c,d\}$ and $R$ be symmetric on $V$ and defined by $R(a,d)=V$, $R(b,d)=\{b,c,d\}$ and 
$R(u,v)=\{u,v\}$ for all the other pairs $u,v\in V$. One can easily check that $R$ is a monotone transit function satisfying $(Ch)$, $(J0)$, $(b1)$, $(a')$ and $(k)$. However, $\mathscr{C}_R$ is not a convex geometry, since the convex set $K=\{a,b\}$ and the points $c,d\notin K$ violate the anti-exchange property. For this notice that $\left\langle K \cup c\right\rangle=V$ because $\{a,b,c\}$ is not a member of $\mathscr{C}_R$.
\end{example}

\begin{corollary}
Let $R$ be a monotone transit function satisfying $(a')$ and $(k)$. Suppose $\mathscr{C}_R$ is a convex geometry, then $R$ satisfies $(b1)$ and $(J0)$.
\end{corollary}

\begin{proof}
The proof is the same as of Theorem \ref{Cg1} with $K=R(x,y)$ for $x,y\in V$. 
\end{proof}

Note that $(m)$, $(Ch)$, $(J0)$ and $(b1)$ need not imply $(a')$ or $(k)$.
\begin{example}
Let $R$ be symmetric on $V=\{a,b,c,d\}$, defined by $R(a,c)=\{a,b,c\}$, $R(a,d)=\{a,b,d\}$, $R(c,d)=\{c,b,d\}$ and $R(u,v)=\{u,v\}$ for all the other pairs $u,v\in V$. It is easy to check that $(m)$, $(k)$, $(Ch)$, $(J0)$ and $(b1)$ hold for $R$ but $(a')$ is violated. Therefore, $\mathscr{C}_R$ is not a convex geometry since $V\notin \mathscr{C}_R$.
\end{example}

\begin{example}
Let $R$ be symmetric on $V=\{a,b,c,d,e\}$, defined by $R(a,e)=\{a,b,c,e\}$, $R(b,d)=\{b,a,d\}$, $R(b,e)=\{b,c,e\}$, $R(c,d)=\{c,a,b,d\}$, $R(d,e)=V$ and $R(u,v)=\{u,v\}$ for all the other pairs $u,v\in V$. Now $R$ satisfies $(m)$, $(a')$, $(Ch)$, $(J0)$ and $(b1)$ and violates $(k)$, since $R(a,e)\cap R(c,d)=\{a,b,c\}$ is not a transit set.  Therefore, $\mathscr{C}_R$ is not a convex geometry since it is not even a convexity.
\end{example}

Theorem \ref{thm:edeljami} states that, for a convexity, the 
 anti-exchange axiom is equivalent to the property that for every convex 
 set $K$ there is a point $p\in V$ such that $K\cup \{p\}$ is convex. 
 The latter can be translated in terms of transit functions.

$(cg)$ For every $x,y\in V$, there exists $z\in R(x,y)$ and $w\notin R(x,y)$ such that $R(x,y)\cup \{w\}=R(w,z)$.

\begin{theorem}\label{thm:cg'}
Let $R$ be a monotone transit function satisfying $(a')$ and $(k)$. Then, $\mathscr{C}_R$ is a convex geometry if and only if $(cg)$ holds for $R$.  
\end{theorem}
\begin{proof}
Already, $\mathscr{C}_R$ is a convexity. Assume that $\mathscr{C}_R$ is convex geometry. Then, for every $K\in \mathscr{C}_R$, there exists $p\in V$ such that $K\cup \{p\}\in \mathscr{C}_R$. Since the convex sets here are exactly transit sets, we can say that, for every $x,y\in V$, there exists $w\notin R(x,y)$ such that $R(x,y)\cup \{w\}=R(p,q)$ for some $p,q\in V$. Clearly, $p,q\in R(x,y)\cup \{w\}$. If both $p,q\in R(x,y)$, then by monotonicity, we arrive at a contradiction. Therefore, $p$ or $q$ must be $w$. Hence, we get $(cg)$.

Conversely, assume that $R$ holds $(cg)$. This means that for every transit set $K$, there is a point $w\in V$ such that $K\cup \{w\}$ is also a transit set. Since the transit sets are precisely the convex sets, we see that $\mathscr{C}_R$ is a convex geometry. 
\end{proof}

\begin{lemma}\label{lem:cg'->a'}
Let $R$ be a transit function satisfying $(cg)$. Then $R$ satisfies $(a')$.
\end{lemma}

\begin{proof}
Consider a transit set $R(x,y)$ for some $x,y\in V$. If $R(x,y)=V$, we have done. If not, applying $(cg)$, we see that there exists $x'\in R(x,y)$ and $y'\notin R(x,y)$ such that $R(x,y)\cup \{y'\}=R(x',y')$. By construction, $R(x,y)\subset R(x',y')$. If $R(x',y')=V$, we are done. Otherwise, there exists $x''\in R(x',y')$ and $y''\notin R(x',y')$ such that $R(x',y')\cup \{y''\}=R(x'',y'')$. Also,  $R(x',y')\subset R(x'',y'')$. If we continue this process, we get transit sets that are nested properly. Since $V$ is finite, this must end at some point. That is, we must get $x^k,y^k$ such that $R(x^k,y^k)=V$ after a finite number of steps. 
\end{proof}

From the above Theorem \ref{thm:cg'} and Lemma \ref{lem:cg'->a'}, we can conclude the following.

\begin{corollary}
Let $R$ be a monotone transit function satisfying $(k)$. Then $\mathscr{C}_R$ is a convex geometry if and only if $R$ holds $(cg)$.  
\end{corollary}

\section {Concluding Remarks}

In this paper, we obtained two partial characterizations 
of convex geometries arising from interval convexities using the axioms on the corresponding associated transit functions by means of Theorems~\ref{cg} and~\ref{Peano-ch}. Moreover, we have analysed the instances of such convex geometries of Theorem~\ref{cg}, but not of Theorem~\ref{Peano-ch}. 
Analysing special instances of Theorem~\ref{Peano-ch} in convex geometries of different settings would also be an interesting problem, which may be pursued in future. 

More challenging problem is to provide a characterization of transit functions using axioms whose associated interval convexities form convex geometries. Towards this, it would be interesting to find an axiom which is stronger than the monotone axiom (m) and weaker than the Peano axiom (P), but we do not know whether such an axiom exists. \\

In Subsection~\ref{C-convexity}, we have discussed the case of the cut-vertex transit function of a graph $G$ and found that the corresponding convexity forms a convex geometry and coincides with the convexity of connected sets (connected alignment) in $G$ and this happens if and only if $G$ is a block graph.  It may be noted that the cut-vertex transit function is also defined on hypergraphs, but we observe that the convexity generated by the cut-vertex transit function may not yield a convex geometry in the case of a hypergraph. We may briefly discuss about the cut-vertex transit function of a hypergraph. For this we need a few definitions on hypergraphs.  

A \emph{hypergraph} $H$ consists of a non-empty set $V$ of vertices and a collection of non-empty subsets $E$ of $V$ ($E\subseteq 2^V$) known as hyperedges. Let $u,v\in V$ be two vertices of $H$. A $u,v$-\emph{path} in $H$ is an alternating sequence $(u=x_1,e_1,x_2,e_2,\dots ,x_n,e_n,x_{n+1}=v)$, of distinct vertices $x_i$ and distinct edges $e_i$ such that $x_i,x_{i+1}\in e_i$ for every $i\in \{1,\dots,n\}$. A hypergraph is \emph{connected} if there exists a $u,v$-path for every pair $u,v\in V$. 

Let $v\in V$ be a vertex of a connected hypergraph $H$. The operation of \emph{strong vertex deletion} of $v$ results in a hypergraph $H-v$ where $v$ and all the edges containing $v$ where removed from $H$. A \emph{strong cut-vertex} is a vertex $x$ such that $H-x$ is not connected \cite{Dewar}. A vertex $x$ is a strong cut-vertex in $H$ if and only if there exists two distinct vertices $u\neq x$ and $v\neq x$ such that every $u,v$-path contains an edge containing $x$. In this case, we say that $x$ \emph{separates} $u$ and $v$ in $H$. 

The cut-vertex transit function $C$ of a hypergraph $H$ with vertex set $V$ is defined as the function $C:V\times V\to 2^V$ given by 
\begin{equation}\label{c1}
C(u,v)=\{x\in V : x \;\text{lies on every $u,v$-path}\}.
\end{equation}
introduced by Duchet \cite{Duchet:84}.  Equivalently, $C(u,v)$ is the set of all strong cut-vertices separating $u$ and $v$ together with $u$ and $v$. By convention, $C(u,u)=\{u\}$ for every $u\in V$. Cut-vertex transit function of graphs and hypergraphs are studied in \cite{Changat:21}.

One can easily verify that the cut-vertex transit function of a graph satisfies $(b1)$, $(Ch)$ and $(J0)$. Therefore, the $C$-convexity generated by the cut-vertex transit function $C$ of a graph is a convex geometry. But this may not be true for hypergraphs in general. 
The hypergraph $H$ in Figure~\ref{C-cg}(i) makes an example of it where  $K=\{c,d\}$ is a convex set, but no $x\in V$ exists such that $K\cup \{x\}$ is convex. Therefore cut-vertex transit function does not yield a convex geometry for $H$. Further, in Figure \ref{C-cg}(ii) and (iii) there are two examples, where the convexity yields by the cut-vertex transit function is a convex geometry for the hypergraph in question. It is easy to see that hypergraph in Figure \ref{C-cg}(ii) contains one strong cut-vertex, while hypergraph in Figure \ref{C-cg}(i) contains two strong cut-vertices. This is the line where we separate the following result and an open problem. 

\begin{figure}[ht]
\begin{center}
\includegraphics{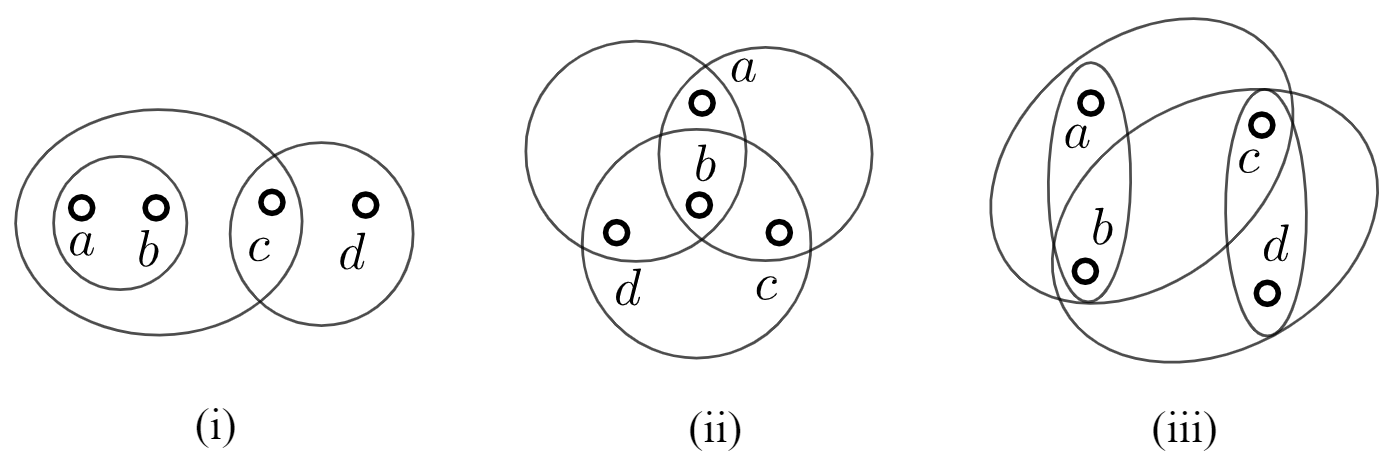}
\end{center}
\caption{Hyperhraphs with with one $(ii)$, with two $(iii)$ and with three $(i)$ strong cut-vertices.}\label{C-cg}
\end{figure}

\begin{proposition}
If a hypergraph $H$ contains at most one strong cut-vertex, then the cut-vertex transit function yields a convex geometry on $H$.
\end{proposition}

\begin{proof} 
Clearly $(J0)$ holds by nothing, since there is at most one strong cut-vertex in $H$. For $(b1)$ $x\in C(u,v)$ and $v\neq x$ gives two possibilities: either $u=x$ or $x$ is a strong cut-vertex. In the first case, clearly $v\notin C(u,x)=\{u\}$ and in the second case again $v\notin C(u,x)$, because $v\neq x$ is not a strong cut-vertex as there is only one strong cut-vertex in $H$. We are left with $(Ch)$. If $y=w$, then $y\in C(u,w)$ and $(Ch)$ holds. If $y=x$, then $y\in C(u,v)$ and $(Ch)$ holds again. Otherwise, $y$ is a strong cut-vertex different than $x$. Since there is only one strong cut-vertex, we have $x=u$ or $x=v$ and $y\in C(u,w)$ or $y\in C(v,w)$, respectively. So, $(Ch)$ is fulfilled and the result follows from Theorem \ref{cg}.  
\end{proof}

\begin{problem}
Describe for which hypergraphs the cut-vertex transit function yields a convex geometry.
\end{problem}

\end{document}